\documentclass[reqno,10pt]{amsart}

\usepackage[utf8]{inputenc}
\usepackage[T1]{fontenc}
\usepackage[english]{babel}
\usepackage{enumerate}
\usepackage{mathtools}
\usepackage{esint}
\usepackage{setspace}
\usepackage{mathrsfs}
\usepackage{braket}
\usepackage{amsmath, amsthm, amsfonts, amssymb,bm}
\usepackage{cite}
\usepackage{bigints}
\usepackage{bbm}
\usepackage{mathabx}
\usepackage{comment}
\usepackage{geometry}
\usepackage{cancel}
\usepackage[normalem]{ulem}
\usepackage{bohr}
\usepackage{enumitem}

\allowdisplaybreaks
\usepackage{subcaption}
\usepackage{graphicx}
\usepackage{booktabs}
\usepackage{caption}
\usepackage{subcaption}

\usepackage{xcolor}
\definecolor{ForestGreen}{RGB}{34,139,34}

\makeatletter

\numberwithin{equation}{section}
\usepackage{amssymb,amsthm,amsmath} 
\usepackage{indentfirst}
\usepackage{xcolor}
\usepackage{graphicx}
\usepackage{url}									
\usepackage{hyperref}

\numberwithin{equation}{section}

\newtheorem{theorem}{Theorem}[section]
\newtheorem{lemma}[theorem]{Lemma}
\newtheorem{corollary}[theorem]{Corollary}
\newtheorem{proposition}[theorem]{Proposition}
\newtheorem{assumption}[theorem]{Assumption}

\newtheorem{definition}[theorem]{Definition}

\def\N{{\mathbb N}}
\def\Z{{\mathbb Z}}
\def\R{{\mathbb R}}
\def\C{{\mathbb C}}

\def\1{{\mathds{1}}}

\def\cD{\mathcal{D}}

\def\cH{{\mathcal H}}

\def\cB{{\mathcal B}}

\def\cL{{\mathcal L}}

\newcommand{\Supp}{{\mathrm{Supp}\,}}

\newtheorem{remark}[theorem]{Remark}

\title[Mixed regularity and sparse grid of electronic system]{Mixed regularity and sparse grid approximations of $N$-body Schr\"odinger evolution equation}

\author{Long Meng}
\address{Long Meng:
Center for Interdisciplinary Applied Mathematics \& Institute of Fundamental and Transdiciplinary Research, Zhejiang University, China}
\email{longmeng@zju.edu.cn}

\author{Dexuan Zhou}
\address{Dexuan Zhou:
School of Mathematical Sciences, Beijing Normal University, Beijing 100875, China}
\email{zhoudexuan@mail.bnu.edu.cn}
  \date{}

\begin{document}


\maketitle

\begin{abstract}
In this paper, we present a mathematical analysis of time-dependent $N$-body electronic systems and establish mixed regularity for the corresponding wavefunctions. Based on this, we develop sparse grid approximations to reduce computational complexity, including a sparse grid Gaussian-type orbital (GTO) scheme. We validate the approach on the Helium atom (${\rm He}$) and Hydrogen molecule (${\rm H}_2$), showing that sparse grid GTOs offer an efficient alternative to full grid discretizations.
\end{abstract}

\section{Introduction}
This paper is devoted to the mathematical study of sparse grid approximations for the following time-dependent $N$-body electronic system in molecular dynamics:
\begin{equation}\label{eq:1.1}
    \begin{cases}
    i\partial_t u(t,x)= H_N(t) u(t,x), & t\in [0,T]=:I_T,\; x=(x_1,\cdots,x_N)\in (\R^3)^N,\\
    u(0,x)=u_0(x),
    \end{cases}
\end{equation}
with
\begin{align}\label{hamiltonian}
    H_N(t)=\sum_{j=1}^N -\frac{1}{2}\Delta_j+\sum_{j=1}^N V(t,x_j)+\sum_{1\leq j<k\leq N} W(x_j,x_k),
\end{align}
where
\begin{align*}
    V(t,x_j)=-\sum_{\mu=1}^M \frac{Z_\mu}{|x_j-a_\mu(t)|}, 
    \qquad 
    W(x_j,x_k)=\frac{1}{|x_j-x_k|}.
\end{align*}
Here $\Delta_j:=\Delta_{x_j}$ denotes the Laplacian operator acting on the variable $x_j$.

\medskip
In physics and quantum chemistry, Eq.~\eqref{eq:1.1} describes the quantum mechanical $N$-body problem, where $N\in \mathbb{N}^+$ electrons interact with $M\in \mathbb{N}^+$ nuclei of total charge $Z=\sum_{\mu=1}^M Z_\mu$ through Coulomb attraction and repulsion. The Hamiltonian \eqref{hamiltonian} acts on wavefunctions with variables $x_1,\ldots,x_N\in \R^3$, representing the coordinates of $N$ electrons. Each moving nucleus $\mu$ with charge $Z_\mu\in \N^+$ is treated as a classical particle located at $a_\mu(t)\in \R^3$ at time $t$. This model serves as the basis for studying dynamical phenomena such as chemical reactions. An analogous setting for the time-dependent Hartree--Fock model coupled with classical nuclear dynamics can be found in \cite{Cances}.

In mathematics, the evolution equation \eqref{eq:1.1} has been well studied. In the time-independent case $H_N(t)\equiv H_{N}(0)$, the Stone theorem ensures the existence and uniqueness of solutions. When $H_N(t)$ depends explicitly on time, the problem becomes much more delicate: by using the Duhamel formula and Strichartz estimates, local-in-time existence and uniqueness of solutions in $C^0(I_T,H^2((\R^3)^N))$ has been proved in \cite{Yajima1} for the one-body problem and in \cite{Yajima2} for the general $N$-body problem (with $T\lesssim (ZN+N^2)^{-2-}$).

In numerical analysis, this electronic configuration space $(\R^3)^N$ is typically high-dimensional for large $N\in \N^+$. Conventional discretizations of partial differential equations by finite differences or finite element methods scale exponentially with respect to the dimension of the configuration space. This is known as the \textit{curse of dimensionality}. 

To overcome this problem, various reduced models have been introduced. Time-dependent (multiconfiguration) Hartree--Fock and time-dependent density functional theory provide effective approximations (see, e.g., \cite{Catto,Chadam,quantumtoclassical-lubich}), however they are not based on a direct discretization of \eqref{eq:1.1}. Other approaches, such as the time-dependent density matrix renormalization group (TD-DMRG) and the time-dependent variational Monte Carlo (TD-VMC) method, have been developed in physics to study $N$-body dynamics (see, e.g., \cite{nys2024ab,baiardi2021electron} and references therein), but face challenges for Coulomb systems or lack rigorous mathematical understanding.  
The Gaussian wave-packet method is also widely used in quantum dynamics (see, e.g., \cite{quantumtoclassical-lubich,quantumdynamics-LasserLubich,moleculardynamics-Hans}), but typically treats \eqref{eq:1.1} within a semiclassical regime.

In this paper, we study the $N$-body problem \eqref{eq:1.1} via sparse grid methods to improve computational efficiency. To this end, we first analyze the mixed regularity of \eqref{eq:1.1}, which then enables the construction of sparse-grid-type approximations. 

\medskip

\subsection{Mixed regularity and sparse grid approximation for eigenvalue problem}

For the eigenvalue problem
\begin{align}\label{eigen}
    H_N(0)u_*=\lambda u_*\qquad \textrm{with}\qquad \lambda<0,
\end{align}
the computational complexity can be reduced by employing sparse grid methods (see, e.g., \cite{Griebel} for numerical implementations and \cite{Yserentant1,Yserentant2,Yserentant3,meng,Yserentant4} for the mathematical foundations), thanks to the fact that the antisymmetry of the wavefunction $u_*$ can improve its regularity. According to Pauli's principle, the antisymmetry of $u_*$ is an inherent physical property of fermions. In this paper, we follow the definition of antisymmetry given in \cite{meng}. 

\begin{definition}[Generalized antisymmetric function]\label{def:anti}
Let $I\subset\{1,\cdots,N\}$. When $|I|>1$, a wavefunction $u$ is \emph{antisymmetric with respect to $I$} if and only if, for any $j,k\in I$, 
\[
u(P_{j,k}x)=-u(x),\qquad P_{j,k} (\cdots,x_j,\cdots, x_k,\cdots):=(\cdots,x_{k},\cdots, x_j,\cdots).
\]
In particular, when $|I|= 1$ or $I=\emptyset$, every wavefunction $u$ is antisymmetric with respect to $I$.
\end{definition}

According to Pauli's principle, the electronic wavefunction (with spin $\{\pm \frac{1}{2}\}$) satisfies the following property.
\begin{proposition}[Antisymmetry of the electronic wavefunction]\label{prop:anti}
For any electronic wavefunction $u$ and for any fixed spin state, there exist two disjoint index sets satisfying $I_1\bigcap I_2=\emptyset$ and $I_1\bigcup I_2=\{1,\cdots,N\}$ such that $u$ is antisymmetric w.r.t. $I_1$ and $I_2$. 
\end{proposition}
Further details on Proposition~\ref{prop:anti} can be found in \cite[Section~1]{Yserentant1}. In this paper we make the following assumption.
\begin{assumption}\label{ass:anti}
The initial datum $u_0$ is antisymmetric w.r.t. the index sets $I_1$ and $I_2$, where $I_1$ and $I_2$ are as given in Proposition~\ref{prop:anti}.
\end{assumption}
Let $I\subset\{1,\ldots,N\}$. The mixed regularity is characterized by the following fractional Laplacian operator:
\begin{align}
    \cL_{I}=\prod_{j\in I}(1-\Delta_j)^{1/2},
\end{align}
which is defined via the Fourier transform (see Section~\ref{sec:notation}). In the special case $I=\emptyset$, we set $\cL_I=\mathbbm{1}$. 
Relying on the antisymmetry of the wavefunction $u_*$, mixed regularity associated with operator $\cL_I$ are studied in \cite{Yserentant1,Yserentant2,Yserentant3}. A typical result can be stated as follows: under Assumption~\ref{ass:anti},
\begin{align*}
    \mathcal{L}_{I_\ell} u_*\in H^1\big((\R^3)^N\big),\qquad \ell=1,2.
\end{align*}
This type of regularity naturally leads to hyperbolic cross space approximations of eigenfunctions. For $R>0$, define
\begin{align}\label{eq:hyper}
    (\mathcal{P}_{R} u)(x):=\int_{(\R^3)^N}\mathbbm{1}_{\mathcal{D}(R)}(\xi)\,
    \mathcal{F}_{x_1,\ldots,x_N}{u}(\xi)\,e^{2\pi i\xi\cdot x}\,d\xi,
\end{align}
where $\mathcal{F}_{x_1,\ldots,x_N}{u}$ denotes the $N$-body Fourier transform given in \eqref{eq:fourier-3N}, and $\mathbbm{1}_{\mathcal{D}(R)}$ is the characteristic function of the hyperbolic cross domain $\mathcal{D}(R)$,
\begin{equation}\label{eq:H-dom}
    \mathcal{D}(R):=\Big\{(\xi_1,\ldots,\xi_N)\in (\R^3)^N \;\Big|\;
    \sum_{\ell=1, 2}\prod_{i\in I_\ell}\big(1+|\xi_i|^{2}\big)^{1/2}\leq R\Big\}.
\end{equation}
Then one obtains the approximation estimate
\begin{align}\label{eq:approx1}
    \|(1-\mathcal{P}_{R}) u_*\|_{L^2((\R^3)^N)}
    \leq R^{-1}\big\|\sum_{\ell=1, 2}\cL_{I_\ell}u_*\big\|_{H^1((\R^3)^N)}
    \leq R^{-1}\sum_{\ell=1,2}\|\cL_{I_\ell}u_*\|_{H^1((\R^3)^N)}.
\end{align}
In numerical analysis, such hyperbolic cross approximations provide the foundation for constructing sparse grid methods, which significantly reduce computational complexity and allow for the treatment of systems with more electrons than standard discretizations (see, e.g., \cite{Griebel}).

\subsection{Mixed regularity and sparse grid approximation of \texorpdfstring{\eqref{eq:1.1}}{}}

We now consider the mixed regularity and sparse grid approximation of \eqref{eq:1.1}. As in the stationary case, the mixed regularity is also characterized by the operator $\cL_{I}$. 
Our main result on mixed regularity (Theorem~\ref{th:mixregularity}) states that, under Assumption~\ref{ass:anti}, for $p=3^-$ and for $T\lesssim (ZN+N^2)^{-2-}$ (see Remark~\ref{eq:T}), the solution $u$ to \eqref{eq:1.1} satisfies
\begin{align}\label{eq:mixedregularity-L2}
    \sum_{\ell=1,2}\|\cL_{I_\ell} u\|_{L^\infty([0,T],L^2((\R^3)^N))}
    \;\leq\;
    \sum_{\ell=1,2}\|\cL_{I_\ell} u\|_{X_{p,T}}
    \;\lesssim_{p}\;
    \sum_{\ell=1,2}\|\cL_{I_\ell} u_0\|_{L^2((\R^3)^N)},
\end{align}
where $X_{p,T}$, defined in Section~\ref{sec:func}, is our functional space used for the evolution problem \eqref{eq:1.1}.

\begin{remark}[Regularity of the initial datum $u_0$]
Under Assumption~\ref{ass:anti}, the condition $\cL_{I_\ell}u_0 \in L^2((\R^3)^N)$, $\ell=1,2$, is natural. For models such as Hartree--Fock, the initial datum $u_0$ can be written as a Slater determinant,
$
u_0(x)=\Big(\bigwedge_{j\in I_1}\phi_j\Big)\otimes\Big(\bigwedge_{j\in I_2}\phi_j\Big),
$
with orbitals $\phi_j \in H^1(\R^3)$. In this case, $\mathcal{L}_{I_\ell}u_0 \in L^2((\mathbb{R}^3)^N)$ holds for $\ell=1,2$.
\end{remark}
With the mixed regularity result established above, we now turn to the sparse grid approximation of \eqref{eq:1.1} under Assumption~\ref{ass:anti}. As a preliminary step, we introduce an abstract framework by defining a truncation operator $\mathcal{P}_R$ on $L^2((\R^3)^N)$. 
\begin{assumption}\label{ass:1.8}
Let $R > 0$. We assume that $\mathcal{P}_R\in \cB(L^2((\R^3)^N))$ satisfies:
\begin{itemize}
    \item (Projector) $0 \leq \mathcal{P}_R^2 = \mathcal{P}_R \leq \mathbbm{1}_{L^2((\R^3)^N)}$,
    \item (Commutativity) $[\mathcal{P}_R,\Delta_j]=0$ for all $j=1,\ldots,N$,
    \item (Approximation) For any $u\in L^2((\R^3)^N)$,
    $
    \|(1-\mathcal{P}_R)u\|_{L^2((\R^3)^N)}
    \leq \frac{1}{R}\Big\|\sum_{\ell=1}^{2}\cL_{I_\ell}u\Big\|_{L^2((\R^3)^N)}.
    $
\end{itemize}
\end{assumption}
Using $\mathcal{P}_R$, we have the following approximation of \eqref{eq:1.1}
\begin{equation}\label{hyper-approx}
    \begin{cases}
        i\partial_t u_R = H_{N,R}(u_R), & t\in I_T,\\
        u_R(0,x) = \mathcal{P}_{R}(u_0)(x),
    \end{cases}
\end{equation}
where
\[
H_{N,R}(u) = \sum_{j=1}^N -\frac{1}{2}\Delta_j u
+ \sum_{j=1}^N \mathcal{P}_{R}(V(t,x_j)u)
+ \sum_{1\leq j<k\leq N} \mathcal{P}_{R}(W(x_j,x_k)u).
\]
Our main result on the sparse grid approximation (Theorem~\ref{th:justif}) states that, under Assumption~\ref{ass:anti}, a solution $u_R=\mathcal{P}_R u_R$ to~\eqref{hyper-approx} exists in $X_{p,\infty}$. Moreover, for $p=3-$ and $T\lesssim (ZN+N^2)^{-2-}$, the following error bound holds:
\begin{align}\label{eq:error-L2}
    \|u-u_R\|_{L^\infty([0,T],L^2((\R^3)^N))}
    \;\leq\; \|u-u_R\|_{X_{p,T}}
    \;\lesssim_p\; \frac{1}{R}\sum_{\ell=1,2}\|\cL_{I_\ell} u_0\|_{L^2((\R^3)^N)}.
\end{align}

In Lemma~\ref{lem:5.3}, we establish an estimate analogous to~\eqref{eq:approx1}, showing that the operator $\mathcal{P}_{R}$ defined in \eqref{eq:hyper} satisfies Assumption~\ref{ass:1.8}. However, this operator corresponds to the continuous frequency domain $\mathcal{D}(R)$ defined in \eqref{eq:H-dom}, which is not directly suitable for computations.  
For numerical purposes, a discrete approximation is required. To this end, we construct a discretized operator $\mathcal{P}_R$ that also satisfies Assumption~\ref{ass:1.8}, and we present a discretized version of \eqref{hyper-approx}. 
In addition, as Gaussian-type orbitals (GTOs) are often preferable in molecular computations (including molecule dynamics), we further propose a sparse grid GTO approximation, motivated by the mixed regularity property.

\subsubsection{Discretization of \eqref{hyper-approx}}
The projector $\mathcal{P}_R$ can be constructed in various ways (see, e.g., \cite{Yserentant-sparegrid,Griebel2,Griebel}). Here we present one example and illustrate how it leads to a discretization of \eqref{hyper-approx} together with the error bound \eqref{eq:error-L2}.

Let $L>0$ and define the function $\phi_{\mathbf{j}}^{(L)}$ by its Fourier transform
\[
\mathcal{F}_y(\phi_{\mathbf{j}}^{(L)})(\xi_y) 
= L^{3/2} e^{-2\pi i L \mathbf{j}\cdot \xi_y},
\quad \xi_y \in [0,L^{-1})^3,\;\; \mathbf{j}\in\mathbb{Z}^3,
\]
which form an orthonormal basis of $L^2([0,L^{-1})^3)$. We extend these functions to $L^2(\mathbb{R}^3)$ by setting
\[
\mathcal{F}_y(\psi_{\mathbf{j}}^{(L)})(\xi_y) 
= \mathcal{F}_y(\phi_{\mathbf{j}}^{(L)})(\xi_y)\,\mathbbm{1}_{[0,L^{-1})^3}(\xi_y).
\]
To obtain a complete system in $L^2(\mathbb{R}^3)$, we consider the shifted functions
\[
\psi_{\mathbf{j},\mathbf{l}}^{(L)}(y) := 
\psi_{\mathbf{j}}^{(L)}(y - L^{-1}\mathbf{l}), 
\quad \mathbf{l}\in \mathbb{Z}^3,
\]
The collection $\{\psi_{\mathbf{j},\mathbf{l}}^{(L)}\}_{\mathbf{j},\mathbf{l}\in\mathbb{Z}^3}$
thus forms an orthonormal basis of $L^2(\mathbb{R}^3)$. Accordingly, an $N$-electron orthonormal basis function is given by
\begin{align}\label{eq:meyer-3N}
  \psi^{(L)}_{\vec{\mathbf{j}},\vec{\mathbf{l}}}(x)
  :=\prod_{\ell =1}^N \psi^{(L)}_{\mathbf{j}_\ell, \mathbf{l}_\ell}(x_\ell),
  \qquad \vec{\mathbf{j}}=(\mathbf{j}_1,\ldots,\mathbf{j}_N),\;
  \vec{\mathbf{l}}=(\mathbf{l}_1,\ldots,\mathbf{l}_N)\in (\mathbb{Z}^3)^N.
\end{align}
Any wavefunction $u(x)\in L^2((\R^3)^N)$ can be written as 
\begin{align}\label{eq:u-decom}
u(x)=\sum_{\vec{\mathbf{j}}, \vec{\mathbf{l}}\in (\mathbb{Z}^3)^N}
  u^{(L)}_{\vec{\mathbf{j}},\vec{\mathbf{l}}}\,
  \psi^{(L)}_{\vec{\mathbf{j}},\vec{\mathbf{l}}}(x),
  \qquad
  u^{(L)}_{\vec{\mathbf{j}}, \vec{\mathbf{l}}}
  = \left<\psi^{(L)}_{\vec{\mathbf{j}},\vec{\mathbf{l}}},\, u\right>.
\end{align}
Based on this basis, we define the truncated projector
\begin{align}\label{eq:P_R}
  \mathcal{P}_R^{(L)} := \sum_{(\vec{\mathbf{j}},\vec{\mathbf{l}})\in \mathcal{D}_{L,R}}
  \left|\psi^{(L)}_{\vec{\mathbf{j}},\vec{\mathbf{l}}}\right>
  \left<\psi^{(L)}_{\vec{\mathbf{j}},\vec{\mathbf{l}}}\right|,
\end{align}
where $\mathcal{D}_{L,R}$ denotes a hyperbolic-cross-type index set used for sparse grid approximation, defined by
\begin{equation}\label{eq:omega_R}
   \mathcal{D}_{L,R} := \Big\{ (\vec{\mathbf{j}}, \vec{\mathbf{l}}) \in (\mathbb{Z}^3)^N \;\Big|\;
   \big(L^{-1}\vec{\mathbf{l}} + ([0,L^{-1})^3)^N\big) \cap \mathcal{D}(R) \neq \emptyset \Big\}.
\end{equation}
Later we will show in Proposition~\ref{prop:P_R}, $\mathcal{P}_R^{(L)}$ satisfies Assumption~\ref{ass:1.8}, and the corresponding solution can be written as
\begin{align}
    u_R^{(L)}(t,x)
    = \mathcal{P}_R^{(L)} u_R^{(L)}(t,x)
    = \sum_{(\vec{\mathbf{j}},\vec{\mathbf{l}})\in \mathcal{D}_{L,R}}
      u^{(L)}_{\vec{\mathbf{j}},\vec{\mathbf{l}}}(t)\,
      \psi^{(L)}_{\vec{\mathbf{j}},\vec{\mathbf{l}}}(x).
\end{align}
Here the coefficients $u^{(L)}_{\vec{\mathbf{j}},\vec{\mathbf{l}}}(t)$ satisfy the evolution system
\begin{align}\label{eq:wavelet-discret}
    \begin{cases}
        i\partial_t u^{(L)}_{\vec{\mathbf{j}},\vec{\mathbf{l}}}(t)
        = \sum_{(\vec{\mathbf{j}}',\vec{\mathbf{l}}')\in \mathcal{D}_{L,R}}
          \left\langle \psi^{(L)}_{\vec{\mathbf{j}},\vec{\mathbf{l}}},
          H_N(t)\psi^{(L)}_{\vec{\mathbf{j}}',\vec{\mathbf{l}}'} \right\rangle
          u^{(L)}_{\vec{\mathbf{j}}',\vec{\mathbf{l}}'}(t), \\[6pt]
        u^{(L)}_{\vec{\mathbf{j}},\vec{\mathbf{l}}}(0)
        = \left\langle \psi^{(L)}_{\vec{\mathbf{j}},\vec{\mathbf{l}}}, u_0 \right\rangle,
    \end{cases}
    \qquad (\vec{\mathbf{j}},\vec{\mathbf{l}})\in \mathcal{D}_{L,R}.
\end{align}
This provides a sparse grid discretization of \eqref{hyper-approx}, and the resulting solution $u_R^{(L)}$ satisfies the error bound~\eqref{eq:error-L2} (see Proposition~\ref{prop:P_R}).

\begin{remark}[Constraints on $\vec{\mathbf{j}}$]
In practical computations, additional constraints on $\vec{\mathbf{j}}$ are required to evaluate \eqref{eq:wavelet-discret} numerically. Since
$\|u_R(t)\|_{L^2((\R^3)^N)} = \|\mathcal{P}_R u_0\|_{L^2((\R^3)^N)},$
it follows that $\|\mathcal{P}_{R,J}^{(L)}u_{R}-u_{R}\|_{L^2((\R^3)^N)}\to 0$ as $J\to \infty$ for any $t\geq 0$, where 
\[
    \mathcal{P}_{R,J}^{(L)} := \sum_{(\vec{\mathbf{j}},\vec{\mathbf{l}})\in \mathcal{D}_{L,R,J}}
      \left|\psi^{(L)}_{\vec{\mathbf{j}},\vec{\mathbf{l}}}\right>
      \left<\psi^{(L)}_{\vec{\mathbf{j}},\vec{\mathbf{l}}}\right|,
      \qquad  
      \mathcal{D}_{L,R,J} := \big\{(\vec{\mathbf{j}},\vec{\mathbf{l}})\in \mathcal{D}_{L,R} : \|\vec{\mathbf{j}}\|_\infty \leq J \big\}.
\]
Therefore, in practice, the infinite index set $\vec{\mathbf{j}}\in (\mathbb{Z}^3)^N$ can be truncated by imposing $|\vec{\mathbf{j}}|_\infty \leq J$ for large $J$.

Refinements of such constraints on $\vec{\mathbf{j}}$ (or on $L$) depend on a more precise understanding of the decay properties of the solution $u$ to \eqref{eq:1.1}, which remains a natural numerical challenge on unbounded domains such as $\mathbb{R}^d$, $d\geq 1$ (see, e.g., \cite{Yserentant-sparegrid,Griebel}). For the eigenvalue problem \eqref{eigen}, such constraints have been derived in \cite[Section~7]{Griebel} by exploiting the exponential decay of eigenfunctions. For the time-dependent problem, favorable decay properties of wavefunctions are often assumed in quantum chemistry. A rigorous analysis of these properties and their numerical implications for the evolution problem will be addressed by the authors in future work.
\end{remark}

\subsubsection{Sparse grid Gaussian-type orbital approximation}\label{sec:1.2.3} 
Since linear combinations of Gaussian-type orbitals (GTOs) can efficiently approximate the singular behavior of electronic wavefunctions near nuclei while preserving the analytical tractability of two-electron integrals, GTOs have become the standard basis in quantum chemistry for eigenvalue problems and also for dynamics (see, e.g., \cite{wozniak2023exploring,li2020real,Abinitio,tuckerman_2002}). Motivated by this, we propose a sparse grid GTO approximation, which provides a promising framework for reducing computational complexity.

The GTO approximation can be regarded as a special case of partial-wave decomposition with given radial functions. In the following, we first recall the general partial-wave expansion without any restriction on the radial part, and then specialize to the Gaussian-type choice used in Gaussian orbital theory.

For any $v\in L^2(\R^3)$ and a fixed $a\in \R^3$, the partial-wave expansion of $v$ takes the form 
\begin{align}\label{eq:partialwave-a}
    v(y)= \sum_{l=0, 1, \cdots, \;}\sum_{m=-l}^l v_{l,m}(r_y)\,Y_{l,m}(\Omega_{y}), 
    \qquad r_y:=|y-a|\in \R^+,\quad 
    \Omega_{y}:=\frac{y-a}{|y-a|}\in \mathbb{S}^2,
\end{align}
where $\{Y_{l,m}\}_{l,m}$ denotes the spherical harmonics, which forms an orthonormal basis of $L^2(\mathbb{S}^2)$ and satisfies
\begin{align}\label{eq:eigen-sphere}
    \Delta_{\mathbb{S}^2} Y_{l,m}=-l(l+1)Y_{l,m},
\end{align}
with $\Delta_{\mathbb{S}^2}$ the Laplace operator on $\mathbb{S}^2$.  

For the $N$-body wavefunction $u\in L^2((\R^N)$, we apply this decomposition to each electron $x_j=(r_j,\Omega_j)\in \R^3$ for $j=1,\ldots,N$. Let ${\pmb l}=(l_1,\ldots,l_N)\in \mathbb{N}_0^N$, ${\pmb m}=(m_1,\ldots,m_N)\in \mathbb{Z}^N$, the angular part of $u$ takes the form
\[
\mathbb{Y}_{{\pmb l},{\pmb m}}(\Omega_1,\ldots,\Omega_N)
   := \bigotimes_{k=1}^N Y_{l_k,m_k}(\Omega_k).
\]
Accordingly, any $v\in L^2((\R^3)^N)$ admits the expansion
\begin{align}\label{eq:partial-wave}
    v(x_1,\ldots,x_N)=\sum_{({\pmb l},{\pmb m})\in \widetilde{\mathcal{D}}}
    v_{{\pmb l},{\pmb m}}(r_1,\ldots,r_N)\,
    \mathbb{Y}_{{\pmb l},{\pmb m}}(\Omega_1,\ldots,\Omega_N),
\end{align}
where 
\[
\widetilde{\mathcal{D}}
:=\Big\{({\pmb l},{\pmb m})\in \mathbb{N}_0^N\times \mathbb{Z}^N : 
l_j\geq 0,\; -l_j\leq m_j\leq l_j,\; j=1,\ldots,N\Big\}.
\]
Based on this decomposition, we introduce the truncated projector
\begin{align}\label{eq:P_R-GTOs}
    \widetilde{\mathcal{P}}_R
    :=\sum_{({\pmb l},{\pmb m})\in \widetilde{\mathcal{D}}_R}
    \big|\phi_{{\pmb l},{\pmb m}}\big\rangle
    \big\langle \phi_{{\pmb l},{\pmb m}}\big|,
\end{align}
where the sparse grid index set is given by
\begin{align*}
      \widetilde{\mathcal{D}}_R^{\mathrm{SG}}
      :=\Big\{({\pmb l},{\pmb m})\in \widetilde{\mathcal{D}} :
      \sum_{\ell=1, 2}\prod_{j\in I_\ell}(l_j+1/2)\leq R\Big\}.
\end{align*}
Proposition~\ref{prop:orbital} then shows that $\widetilde{\mathcal{P}}_R u$ provides a good approximation of the solution $u$ to \eqref{eq:1.1}, namely
\begin{align}\label{eq:GTOs}
   \sup_{t\in [0,T]} \|(1-\widetilde{\mathcal{P}}_R)u(t)\|_{\mathcal{V}_{I_1,I_2}}
   \;\lesssim\; R^{-1}\sum_{\ell=1,2} \|\mathcal{L}_{I_\ell}u_0\|_{H^1((\mathbb{R}^3)^N)},
\end{align}
where $\mathcal{V}_{I_1,I_2}$ is the functional space defined in \eqref{eq:V12}. We also point out that for the eigenvalue problem \eqref{eigen}, there exists a similar result \cite{Yserentant4} but in different functional spaces.

In the atomic-orbital framework, localized Gaussian-type orbitals (GTOs) centered at the nuclear positions 
\(\{a_\mu\}_{\mu=1}^{M}\) are of the form
\[
\psi_{\mu,n,l,m}(y) 
= f_{n,l,m}(t,|y-a_{\mu}(t)|)\,Y_{l,m}\!\big(\Omega_{y,a_\mu}\big),
\]
where 
\[
f_{n,l,m}(t,s):=\sum_{k} c^{(n)}_{l,m,k}(t)\, s^{l}
  e^{-\zeta^{(n)}_{l,m,k}\, s^2}, 
  \qquad \Omega_{y,a_\mu}:=\frac{y-a_\mu}{|y-a_\mu|}.
\]
Here the contraction coefficients satisfy $c^{(n)}_{l,m,k}\in\R$ and the exponents $\zeta^{(n)}_{l,m,k}>0$, while the index $n$ labels different contracted orbital sets (i.e., linear combinations of primitive Gaussians). 

In this representation, any $N$-body wavefunction $u$ can be expanded in terms of GTOs as
\begin{align}\label{eq:GTOs-N}
    u(t,x) 
    = \sum_{({\pmb l},{\pmb m})\in \widetilde{\mathcal{D}}}
      \sum_{\pmb n}\prod_{j=1}^N
      \left(\sum_{\mu=1}^M \psi_{\mu,n_j,l_j,m_j}(t,x_j)\right),
\end{align}
where $\pmb n=(n_1,\ldots,n_N)$ is a multi-index labeling the different radial functions corresponding to each $(l_j,m_j)$. 
In Gaussian orbital theory, however, $\pmb n$ ranges over a fixed finite set determined by the chosen basis (e.g., in the cc-pVDZ basis, Helium has two $s$-type orbitals with $n=1,2$ for $l=0$). 
Hence, in the subsequent sparse grid truncations, no further restriction is imposed on $\pmb n$, as the number of contracted radial functions is already finite and fixed by the chosen basis set in Gaussian orbital theory.

We now explain how the sparse grid \eqref{eq:P_R-GTOs} is extended to  GTO approximation. 
For the atomic case ($M=1$), by \eqref{eq:GTOs} the construction is straightforward: the wavefunction $u(t,x)$ can be approximated by
\begin{align}\label{eq:GTOs-N-1}
  u(t,x)\;\approx\; u^{\rm SG}(t,x):=\widetilde{\mathcal{P}}_R u(t,x)
  = \sum_{({\pmb l},{\pmb m})\in \widetilde{\mathcal{D}}_R^{\rm SG}}\sum_{\pmb n}
    \prod_{j=1}^N \psi_{1,n_j,l_j,m_j}^{(p)}(t,x_j).
\end{align}
For molecular systems ($M>1$), the situation is more involved, since the angular variables 
\(\Omega=(y-a_1(t))/|y-a_1(t)|\) and 
\(\Omega'=(y-a_2(t))/|y-a_2(t)|\) are centered at different nuclei. 
Therefore, \eqref{eq:GTOs} cannot be applied directly. However, as molecular orbitals in quantum chemistry are typically expressed as linear combinations of atomic orbitals, it is natural to extend the sparse grid approximation from the atomic to the molecular case by applying the same sparse grid truncation in the angular indices:  
\begin{align}\label{eq:GTOs-N-2}
     u(t,x)\;\approx\; u^{\rm SG}(t,x)
     := \sum_{({\pmb l},{\pmb m})\in \widetilde{\mathcal{D}}_R^{\rm SG}}\sum_{\pmb n}
     \prod_{j=1}^N\left(\sum_{\mu_j=1}^M \psi_{\mu_j,n_j,l_j,m_j}^{(p)}(t,x_j)\right).
\end{align}

As an illustration, we apply this sparse grid GTO approximation to the Helium atom (${\rm He}$) and the Hydrogen molecule (${\rm H}_2$), and demonstrate its efficiency; see Figure~\ref{fig:sg-vs-fg}.

{\bf Organisation of the paper.} 
This paper is organised as follows. 
In Section~\ref{sec:mainresult}, we set up the problem, introduce notation (Section~\ref{sec:notation}) and functional spaces (Section~\ref{sec:func}), and state our main results on existence, mixed regularity (Section~\ref{sec:existence}), and sparse grid approximations (Section~\ref{sec:wavelet} for wavelet-based sparse grid approximation and Section~\ref{sec:gaussian} for Gaussian-type orbital sparse grid approximation). A numerical application to the He atom and H$_2$ molecule is also presented (Section~\ref{sec:gaussian}). 
In Section~\ref{sec:3}, we introduce the Strichartz estimates, Hardy inequalities and Sobolev inequalities in our functional spaces. 
The proofs of existence and mixed regularity   (Theorems~\ref{th:existence}--\ref{th:mixregularity}) are given in Section~\ref{sec:4}, and the justification of the error bound (Theorem~\ref{th:justif}) is provided in Section~\ref{sec:5}.

\section{Set-up and main results}\label{sec:mainresult}
In this section, we first introduce the notation and functional spaces. Then we state the main existence and mixed regularity results. Building on this, we develop sparse grid approximations based on wavelets and the Gaussian-type orbital. The section ends with an application to the He atom and H$_2$ molecule.

\subsection{Notations}\label{sec:notation} 
To avoid ambiguity, we first clarify the notations used throughout the paper.

We write $a \lesssim b$ to mean that there exists a constant $C>0$, independent of $N$ and $Z$, such that $a \leq C b$. 
Similarly, $a \lesssim_p b$ means that there exists a constant $C(p)>0$, depending only on $p$, such that $a \leq C(p)b$. 
Whenever these symbols are used, the implicit constants $C$ (resp. $C(p)$) are always independent of $N$ and $Z$.

Next, we fix our convention for the Fourier transform. For $f\in L^2(\R^3)$ and $g\in L^2((\R^3)^N)$, we define
\begin{align}\label{eq:fourier-3}
    \mathcal{F}_{y}(f)(\xi_y)
    := \int_{\R^3} f(y)\, e^{-2\pi i \xi_y\cdot y}\, dy,
\end{align}
and
\begin{align}\label{eq:fourier-3N}
    \mathcal{F}_{x_1,\cdots,x_N}(g)(\xi)
    := \mathcal{F}_{x_N}\circ\cdots\circ \mathcal{F}_{x_1}(g)(\xi),
    \qquad \xi=(\xi_1,\cdots,\xi_N),\;\; \xi_k\in\mathbb{R}^3,\; k=1,\cdots,N.
\end{align}
The subscript $y$ or $x_j$ indicates the variable on which the Fourier transform acts. 

For any $I\subset\{1,\cdots,N\}$, we define the operator
\begin{align}
    \cL_{I} := \prod_{j\in I}(1-\Delta_j)^{1/2}
\end{align}
in the Fourier transform sense
\begin{align*}
    \mathcal{F}_{x_1,\cdots,x_N}(\mathcal{L}_{I} g)(\xi)
    := \prod_{i\in I}\big(1+|2\pi \xi_i|^2\big)^{1/2}\,
       \mathcal{F}_{x_1,\cdots,x_N}(g)(\xi).
\end{align*}
Finally, we recall some standard notations for Strichartz estimates. For any $2\leq p\leq 6$, we set
\begin{enumerate}
    \item $p'$ to be the conjugate exponent of $p$, i.e.,
    \begin{align}\label{eq:p'}
        \frac{1}{p'}+\frac{1}{p}=1;
    \end{align}
    \item $\theta_p$ by
    \begin{align}\label{eq:theta-p}
        \frac{2}{\theta_p}:=3\left(\frac{1}{2}-\frac{1}{p}\right).
    \end{align}
\end{enumerate}
The pair $(p,\theta_p)$ is called Schr\"odinger admissible on $\R^3$. In particular, $(6,2)$ is the endpoint admissible pair. 
The notation $\theta_p'$ is also defined by \eqref{eq:p'}, i.e., $ \frac{1}{\theta_p}+\frac{1}{\theta_p'}=1.$

\subsection{Function spaces}\label{sec:func}
For the $N$-body problem \eqref{eq:1.1}, a main challenge is to define functional spaces that can handle the Coulomb singularities. We first introduce the antisymmetric subspaces $\cH_I$ and its mixed regularity variant $H^1_{I,\rm mix}$. We then define the $L^{p,2}_{x_j}$ and $L^{p,2}_{j,k}$ spaces needed to treat electron–nucleus and electron–electron interactions, and finally present the full evolution space $X_{p,T}$ and its mixed regularity counterpart $X_{I,p,T}^1$.

Let $\cH=L^2((\R^3)^{N})$. For any $I\subset \{1,\dots,N\}$, define the Hilbert space of $I$-antisymmetric wavefunctions as
\begin{align}
    \cH_{I}:=\{g\in \cH:\; g \text{ is antisymmetric with respect to } I\}.
\end{align}
For mixed regularity, we set
\[
H^1_{I,\rm mix}:=\{g\in \cH_I:\; \cL_I g\in \cH\},
\qquad 
\|g\|_{H^1_{I,\rm mix}}:=\|\cL_I g\|_{\cH}.
\]
For spaces $A$ and $B$, the space of bounded operators $\cB(A,B)$ is endowed with the operator norm
\[
\|T\|_{\cB(A,B)}:=\sup_{\|u\|_{A}=1}\|Tu\|_B,
\]
with the shorthand $\cB(A):=\cB(A,A)$.

\medskip
We now introduce the functional spaces used for the evolution problem \eqref{eq:1.1}.  
To treat the electron–nucleus potential $V(\cdot,x_j)$, for $1<p<\infty$, we define  
\[
L^{p,2}_{x_j}=L^p(\R^3_{x_j},L^2((\R^3)^{N-1}))
\]
with norm
\[
\|g\|^p_{L^{p,2}_{x_j}}
=\int_{\R^3_{x_j}}\Bigg(\int_{(\R^3)^{N-1}}
   |g|^2\,dx_1\cdots\widehat{dx_j}\cdots dx_N\Bigg)^{p/2}dx_j,
\]
where $\widehat{dx_j}$ indicates omission of the $j$-th variable. We also write $L^{p,2}_{j}$ for $L^{p,2}_{x_j}$.  For the electron–electron potential $W(x_j,x_k)$, we first change variables to  
\[
r_{j,k}:=\tfrac{1}{2}(x_j-x_k),\qquad R_{j,k}:=\tfrac{1}{2}(x_j+x_k),
\]
and introduce the unitary operator $\mathcal{R}_{j,k}$ by
\begin{align}\label{eq:0}
  \MoveEqLeft  \mathcal{R}_{j,k}g(r_{j,k},R_{j,k},x_1,\cdots,x_{j-1},x_{j+1},\cdots,x_{k-1},x_{k+1},\cdots,x_N)\notag\\
  &=g(\cdots,x_{j-1},(r_{j,k}+R_{j,k}),x_{j+1},\cdots,x_{k-1},(R_{j,k}-r_{j,k}),x_{k+1},\cdots).
\end{align}
We then define
\[
L^{p,2}_{j,k}=L^p(\R^3_{r_{j,k}},L^2((\R^3)^{N-1}))
\]
with norm
\[
\|g\|^p_{L^{p,2}_{j,k}}
=\int_{\R^3_{r_{j,k}}}\Bigg(\int_{(\R^3)^{N-1}}
   |\mathcal{R}_{j,k}g|^2\,dR_{j,k}\,dx_1\cdots\widehat{dx_j}\cdots\widehat{dx_k}\cdots dx_N\Bigg)^{p/2}dr_{j,k}.
\]
 Obviously,
 \begin{align}\label{Lp2-rjk}
\|g\|_{L^{p,2}_{j,k}}=\|\mathcal{R}_{j, k} g\|_{L^{p,2}_{r_{j,k}}}.     
 \end{align}
For future convenience, we will use the unified notation $L^{p,2}_D$ for $D\subset{1,\dots,N}$ with $1\le |D|\le 2$. 
More precisely, if $D=\{j\}\subset \{1,\cdots,N\}$, we have $ L^{p,2}_D=L^{p,2}_j;$  if $D=\{j,k\}\subset \{1,\cdots,N\}$, we have $L^{p,2}_D=L^{p,2}_{j,k}.$

For the time-dependent problem \eqref{eq:1.1}, we work in the functional space
\[
X_{p,T}= L_t^\infty(I_T,\cH)\,\bigcap_{\substack{D\subset \{1,\dots,N\}\\ 1\le |D|\le 2}}
   L_t^{\theta_{p}}(I_T,L_D^{p,2}), \qquad p>2,
\]
endowed with the norm
\[
\|u\|_{X_{p,T}}
=\max\left\{\|u\|_{L_t^\infty(I_T,\cH)},\;
 \max_{D\subset \{1,\dots,N\},\; 1\le |D|\le 2}
   \|u\|_{L_t^{\theta_p}(I_T,L_D^{p,2})}\right\}.
\]
In addition, for any $D\subset\{1,\dots,N\}$ with $1\le |D|\le 2$, the dual space of $L^{\theta_{p}}_t(I_T,L^{p,2}_D)$ is $L^{\theta'_{p}}_t(I_T,L^{p',2}_D)$.

Concerning mixed regularity for the time-dependent problem \eqref{eq:1.1}, we also need the following functional space
\begin{align*}
    X_{I,p,T}^1=\{u\in X_{p,T}\bigcap L_t^\infty(I_T,\cH_I);\; \cL_I u\in X_{p,T}\}
\end{align*}
with the norm $ \|u\|_{X_{I,p,T}^1}=\|\cL_I u\|_{X_{p,T}}$.

\subsection{Existence and mixed regularity}\label{sec:existence} 
We now state our main result on existence and regularity, under the following assumption.
\begin{assumption}\label{ass}
Assume $\alpha$, $p>0$, and $T>0$ satisfy:
\begin{enumerate}
    \item\label{ass1} $0<\alpha<\tfrac{1}{2}$ and $\tfrac{6}{3-2\alpha}<p<6$;
    \item\label{ass2} $1/\theta_{p}<1/\theta'_{\widetilde{p}}$ for some $\tfrac{6}{1+2\alpha}<\widetilde{p}< 6$;
    \item\label{ass3} $C_{T,1}(Z+N)NT^{1/\theta'_{\widetilde{p}}-1/\theta_p}<\tfrac{1}{2}$, 
    with $C_{T}:=\max\{C_{T,1},C_{T,2},C_{T,3}\}\geq 1$. Here $C_{T,1}$, $C_{T,2}$ and $C_{T,3}$ are constants only dependent on $\alpha, p,\widetilde{p}$ given by \eqref{eq:CT1}, \eqref{eq:CT2} and \eqref{eq:CT3} respectively.
\end{enumerate}
\end{assumption}
\begin{remark}[Nonemptiness of Assumption \ref{ass} and estimate on $T$]\label{eq:T}
In Assumption \ref{ass}, we can take $\alpha=\frac{1}{2}-$, $p=3-$ and $\widetilde{p}=3+$. With this choice, $1/\theta'_{\widetilde{p}}-1/\theta_p= \frac{1}{2}-$. Thus Assumption \ref{ass} is not empty and $T\lesssim (ZN+N^2)^{-2-}$.
\end{remark}

Concerning the existence of solutions to \eqref{eq:1.1}, we have
\begin{theorem}[Existence of solutions]\label{th:existence}
Let $a_{\mu}\in L^\infty(\R)$. For every $u_0\in \cH$, the problem \eqref{eq:1.1} has a unique global-in-time solution $u\in X_{p,\infty}$. Furthermore, under Assumption~\ref{ass} on $p$ and $T$, we have
\begin{align}\label{eq:u-solution}
    \|u\|_{X_{p,T}} \lesssim_{p} \|u_0\|_{\cH}.
\end{align}
\end{theorem}
This result can also be found in \cite{Yajima2} in a more general setting. 
Here we modify the proof to make it consistent with Theorems~\ref{th:mixregularity} and Theorem~\ref{th:justif}; a detailed proof is presented in Section~\ref{sec:exist}. 

\medskip

Concerning the mixed regularity of solutions to \eqref{eq:1.1}, we have
\begin{theorem}[Mixed regularity]\label{th:mixregularity}
Let $I\subset\{1,\dots,N\}$ and  $a_{\mu}\in L^\infty(\R)$. For every $u_0\in H^1_{I,\rm mix}$, under Assumption~\ref{ass}, the problem \eqref{eq:1.1} has a unique solution $u\in X^1_{I,p,T}$ with
\begin{align}\label{eq:u-mix}
    \|u\|_{X^1_{I,p,T}} \lesssim_{p} \|u_0\|_{H^1_{I,\rm mix}}.
\end{align}
\end{theorem}
The proof of Theorem~\ref{th:mixregularity} is given in Section~\ref{sec:mix}.

\subsection{Sparse grid approximation and discretization}\label{sec:wavelet}
We now justify the sparse grid approximation \eqref{hyper-approx} and consider its discretization. The main result of this subsection states 
\begin{theorem}[Sparse grid approximation]\label{th:justif}
Let $a_{\mu}\in L^\infty(\R)$, and $I_1,I_2$ be as in Proposition~\ref{prop:anti}. Let $\mathcal{P}_R$ be a projector satisfying Assumption~\ref{ass:1.8}. For every $u_0\in H_{I_1,\rm mix}^1\cap H_{I_2,\rm mix}^1$,
the problem \eqref{hyper-approx} has a unique global-in-time solution $u_R\in X_{p,\infty}$. Furthermore, $\mathcal{P}_R u_R = u_R$ and under Assumption~\ref{ass},
\begin{align}\label{error-bound}
    \|u-u_R\|_{X_{p,T}}\lesssim_{p} \frac{1}{R}\sum_{\ell=1,2}\|u_0\|_{H_{I_\ell,\rm mix}^1}.
\end{align}
\end{theorem}
The proof of Theorem~\ref{th:justif} is given in Section~\ref{sec:5}.  
We now consider the discretized projector $\mathcal{P}_R^{(L)}$ defined in \eqref{eq:P_R}. 
With the following 
proposition, we can see that $\mathcal{P}_R^{(L)}$ is a direct application of Theorem \ref{th:justif}. 
\begin{proposition}\label{prop:P_R}
The projector $\mathcal{P}_R^{(L)}$ satisfies Assumption~\ref{ass:1.8}, 
and Theorem~\ref{th:justif} holds for the problem \eqref{eq:wavelet-discret}.
\end{proposition}
\begin{proof}
Since $(\psi^{(L)}_{\vec{\mathbf{j}},\vec{\mathbf{l}}}(x))_{\vec{\mathbf{j}},\vec{\mathbf{l}}\in (\Z^3)^N}$ forms an orthonormal basis of $L^2((\R^3)^N)$, $\mathcal{P}_R^{(L)}$ is a projector.   
Moreover, as $\Supp\big(\mathcal{F}(\psi_{\vec{\mathbf{j}},\vec{\mathbf{l}}}^{(L)})\big)\cap 
\Supp\big(\mathcal{F}(\psi_{\vec{\mathbf{j}}',\vec{\mathbf{l}}'}^{(L)})\big)=\emptyset$ for $\vec{\mathbf{l}}\neq \vec{\mathbf{l}}'$, we have
\begin{align*}
\Delta_j
= \sum_{\vec{\mathbf{l}}}\sum_{\vec{\mathbf{j}},\vec{\mathbf{j}}'}
   \left<\psi^{(L)}_{\vec{\mathbf{l}},\vec{\mathbf{j}}'}, \Delta_j \psi^{(L)}_{\vec{\mathbf{l}},\vec{\mathbf{j}}}\right>
   \left|\psi^{(L)}_{\vec{\mathbf{l}},\vec{\mathbf{j}}'}\right>\!\left<\psi^{(L)}_{\vec{\mathbf{l}},\vec{\mathbf{j}}}\right|,
\quad j=1,\dots,N,
\end{align*}
which implies $[\Delta_j,P_R^{(L)}]=0$ for $j=1,\cdots,N$. Finally, by the definition of $\mathcal{D}_{L,R}$ in \eqref{eq:omega_R} and of $\psi_{\vec{\mathbf{j}},\vec{\mathbf{l}}}$, we have 
\[
   \Supp\!\left(\mathcal{F}((1-\mathcal{P}_{R}^{(L)})u)\right)\cap \mathcal{D}(R)=\emptyset,
\]
so that as for \eqref{eq:approx1},
\begin{align*}
    \|(1-\mathcal{P}_{R}^{(L)})u\|_{\cH}
    &= \|\mathcal{F}((1-\mathcal{P}_{R}^{(L)})u)\|_{\cH} \le \frac{1}{R}\left\|(1-\mathcal{P}_{R}^{(L)})\sum_{\ell=1,2}\cL_{I_\ell}u\right\|_{L^2((\R^3)^N)}\le \frac{1}{R}\left\|\sum_{\ell=1,2}\cL_{I_\ell} u\right\|_{L^2((\R^3)^N)}.
\end{align*}
Hence $\mathcal{P}_R^{(L)}$ satisfies Assumption~\ref{ass:1.8}. Applying $\mathcal{P}_R^{(L)}$ to \eqref{hyper-approx} yields \eqref{eq:wavelet-discret}, and thus Theorem~\ref{th:justif} holds.
\end{proof}

\subsection{Sparse grid Gaussian-type orbital approximation and numerical experiments}\label{sec:gaussian}
Recall from \eqref{eq:P_R-GTOs} that the operator $\widetilde{\mathcal{P}}_R$ denotes 
the sparse grid projector based on Gaussian-type orbitals (GTOs). 
The following proposition shows that $\widetilde{\mathcal{P}}_R$ provides a good approximation of $u(t,x)$ in the following sense:
\begin{proposition}[Error estimate for the sparse grid GTO approximation]\label{prop:orbital}
For any $v\in H_{I_1,\rm mix}^1\cap H_{I_2,\rm mix}^1$, the following estimate holds:
\begin{align}
     \|v-\widetilde{\mathcal{P}}_R v\|_{\mathcal{V}_{I_1,I_2}}
     \;\lesssim\;  \frac{1}{R}\Big(\|v\|_{H_{I_1,\rm mix}^1}+\|v\|_{H_{I_2,\rm mix}^1}\Big),
\end{align}
where
\begin{align}\label{eq:V12}
\|v\|_{\mathcal{V}_{I_1,I_2}}
  := \min_{\ell=1,2}\left\| \Big(\prod_{j\in I_\ell}|x_j-a|^{-1}\Big)v\right\|_{\cH}
\end{align}
with $a$ being given in \eqref{eq:partialwave-a}. In particular, under the same assumptions as in Theorem~\ref{th:mixregularity}, the solution $u$ to \eqref{eq:1.1} satisfies
\begin{align}\label{eq:u-PRu-GTOs}
     \sup_{t\in [0,T]} \|(1-\widetilde{\mathcal{P}}_R) u(t)\|_{\mathcal{V}_{I_1,I_2}}
     \;\lesssim\; \frac{1}{R}\Big(\|u_0\|_{H_{I_1,\rm mix}^1}+\|u_0\|_{H_{I_2,\rm mix}^1}\Big).
\end{align}
\end{proposition}
\begin{proof}
The proof relies on the partial-wave decomposition together with an improved Hardy inequality. 
First, recall that under the partial-wave decomposition \eqref{eq:partialwave-a} and \eqref{eq:partial-wave} we have
\begin{align}\label{eq:decom-partialwave}
    \|v(y)\|_{L^2(\R^3)}
    = \sum_{l=0,1,\dots}\;\sum_{m=-l}^{l}
      \|r_y v_{l,m}(r_y)\|_{L^2(\R^+)},\quad  \|v\|_{L^2((\R^3)^N)}
    = \sum_{({\pmb l},{\pmb m})\in \widetilde{\mathcal{D}}}
     \left\|\big(\prod_{j=1}^N r_j\big) v_{{\pmb l},{\pmb m}}\right\|_{L^2((\R^+)^N)}
\end{align}
and for any function of the form $\psi(y)= f(r_y)Y_{l,m}(\Omega_y)$, the improved Hardy inequality gives
\begin{align}\label{eq:hardy-partialwave}
    \|\nabla \psi\|_{L^2(\R^3)}
    \;\ge\; \Big(l+\tfrac{1}{2}\Big)\||\cdot-a|^{-1}\psi\|_{L^2(\R^3)}
    = \Big(l+\tfrac{1}{2}\Big)\|f\|_{L^2(\R^+)} .
\end{align}
Combining \eqref{eq:partial-wave}, \eqref{eq:decom-partialwave}, and \eqref{eq:hardy-partialwave}, we obtain
\begin{align*}
  \|v-\widetilde{\mathcal{P}}_R v\|_{\mathcal{V}_{I_1,I_2}}
  &= \min_{\ell=1,2} \sum_{(\pmb{l},\pmb{m})\notin \widetilde{\cD}_R}
      \Big\|\Big(\prod_{j\in I_\ell^c} r_j \Big) v_{\pmb{l},\pmb{m}}\Big\|_{L^2((\R^+)^N)} \\
  &\le \frac{1}{R}\min_{\ell=1,2}\sum_{(\pmb{l},\pmb{m})\notin \widetilde{\cD}_R}
      \Big\|\Big(\prod_{j\in I_\ell^c} r_j\Big)
      \Big(\sum_{\ell'=1,2}\prod_{k\in I_{\ell'}}(l_k+1/2)\Big)
      v_{\pmb{l},\pmb{m}}\Big\|_{L^2((\R^+)^N)} \\
  &\le \frac{1}{R}\sum_{\ell=1,2}\sum_{(\pmb{l},\pmb{m})\notin \widetilde{\cD}_R}
      \Big\|\Big(\prod_{j\in I_\ell^c} r_j\Big)
      \Big(\prod_{k\in I_\ell}(l_k+1/2)\Big)
      v_{\pmb{l},\pmb{m}}\Big\|_{L^2((\R^+)^N)} \\
  &\le \frac{1}{R}\Big(\|v\|_{H_{I_1,\rm mix}^1}+\|v\|_{H_{I_2,\rm mix}^1}\Big),
\end{align*}
where in the first inequality we used the fact that 
$
1 \;\le\; \frac{1}{R}\sum_{\ell'=1,2}\prod_{j\in I_{\ell'}}(l_j+1/2)$ 
for all $(\pmb{l},\pmb{m})\notin \widetilde{\cD}_R$, 
and the last inequality follows from \eqref{eq:decom-partialwave} and \eqref{eq:hardy-partialwave}.

Finally, under the assumptions of Theorem~\ref{th:mixregularity}, estimate 
\eqref{eq:u-PRu-GTOs} follows from \eqref{eq:u-mix} together with the fact that  
$X_{I_\ell,p,T}^1 \subset L^\infty([0,T],H^1_{I_\ell,\rm mix})$ for $\ell=1,2$.
\end{proof}

{\bf Numerical experiments.} 
We now illustrate the effect of sparse-grid truncation in a time-dependent setting 
for the helium atom (${\rm He}$) and hydrogen molecule (${\rm H}_2$) with $N=2$, 
comparing the sparse grid (SG) Gaussian-type orbital approximation \eqref{eq:GTOs-N-2} 
with a full grid (FG) truncation of the form
\[
     u(t,x)\;\approx\; \sum_{({\pmb l},{\pmb m})\in \widetilde{\mathcal{D}},\;\|{\pmb l}\|_\infty \leq R}\sum_{\pmb n}
     \prod_{j=1}^N\Big(\sum_{\mu_j=1}^M \psi_{1,n_j,l_j,m_j}^{(p)}(t,x_j)\Big).
\]
The one-electron basis is taken as a fixed Gaussian atomic orbital (AO) basis frozen at the $t=0$ geometry: 
AO centers and contraction coefficients remain unchanged during propagation. The electronic wavefunction is propagated in real time with a uniform time step
\[
\Delta t = 0.001, 
\qquad n_{\text{steps}} = 1000, 
\qquad T = n_{\text{steps}} \Delta t = 1.0.
\]
For simplicity, the nuclei are modeled as fixed point charges; that is, their positions remain unchanged during the simulation and only the electronic degrees of freedom are propagated dynamically. This corresponds to the standard Born–Oppenheimer approximation, where the nuclei are assumed to move much more slowly than the electrons. 

For the He atom, the nucleus is fixed at the origin with 65 contracted AOs centered there, and the two electrons are initially placed in Gaussian-type orbitals of the form $\exp(-0.5 r^2)$, one centered at the origin $(0.0)$ and the other slightly shifted to $0.2$ in position (i.e., $\exp(-0.5 (r-0.2)^2)$), both spin-up. For the H$_2$ molecule, the two nuclei are fixed along the $z$-axis at a bond distance of $R=1.4~\text{Bohr}$, with each hydrogen contributing 65 contracted AOs, and the two electrons are initially placed in Gaussian-type orbitals of the form $\exp(-0.5 r^2)$, centered at the two nuclei and both spin-up. In both cases, the resulting CI vector is projected onto the chosen basis set, ensuring that the initial wavefunction has the desired decay properties.

For each truncation type and radius $R$, we measure the error
\[
\mathcal{E}_R^{\mathrm{type}}(t)
=\big\|u^{(R,\mathrm{type})}(t)-u_{\mathrm{full}}(t)\big\|_{H^1((\R^3)^2)},
\]
and 
\[
\Delta E_R^{\mathrm{type}}(t)
= \big|\,E^{(R,\mathrm{type})}(t)-E_{\mathrm{full}}(t)\,\big|
:=\Big|\left<u^{(R,\mathrm{type})}(t),H_N u^{(R,\mathrm{type})}(t) \right>
- \left<u_{\mathrm{full}}(t),H_N(t)u_{\mathrm{full}}(t)\right>\Big|,
\]
where ``type'' denotes either the \emph{sparse grid} (SG) or the \emph{full grid } (FG) truncation and $u_{\mathrm{full}}$ is computed with a full grid using 
$\|\pmb l\|_\infty=7$, corresponding to a determinant-space dimension of 
2080 for He and 8385 for H$_2$. 
Here, the term ``dimension’’ (or ``degrees of freedom’’) refers to the size of the determinant space, that is, the number of Slater determinants constructed from the chosen Gaussian orbitals used to represent the wavefunction $u(t,x)$.

We characterize the largest error for the simulation period $t\in [0,T]$ by
\begin{align*}
   \max_{t\in [0,T]} \mathcal{E}_R^{\mathrm{type}}(t)\qquad \mbox{and}\qquad  \max_{t\in [0,T]}\Delta E_R^{\mathrm{type}}(t),
\end{align*}
and report these quantities in Figure~\ref{fig:sg-vs-fg}. Both errors are plotted as functions of the degrees of freedom to assess how effectively the sparse grid reduces the determinant space while maintaining accuracy. 
\begin{figure}[htbp]
  \centering
  \begin{subfigure}{0.9\linewidth}
  \centering
      \includegraphics[width=0.9\linewidth]{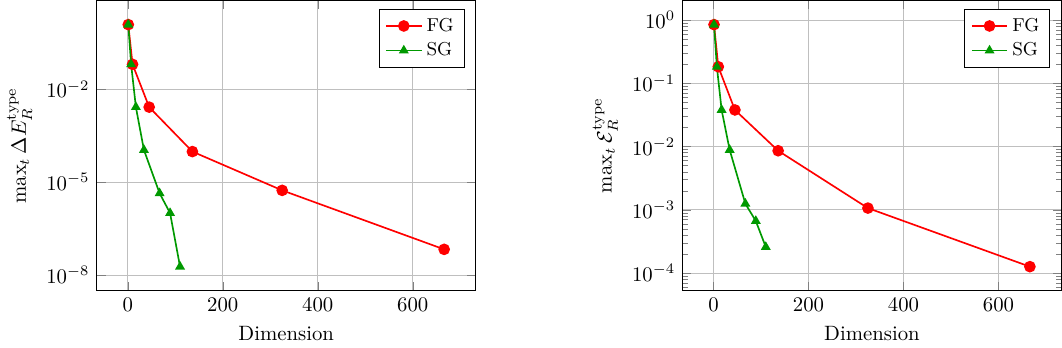}
      \caption{Helium atom (${\rm He}$).}
  \end{subfigure}
   \vspace{1em} 
  \centering
\begin{subfigure}{0.9\linewidth}
  \centering
     \includegraphics[width=0.9\linewidth]{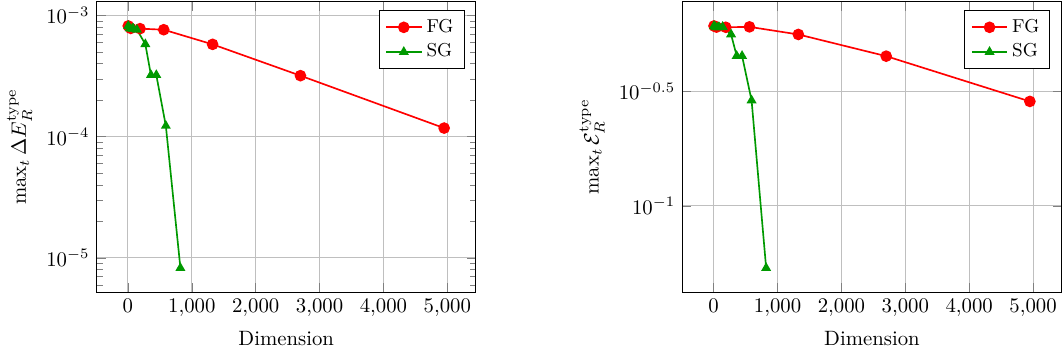}
      \caption{Hydrogen molecule (${\rm H}_2$).}
  \end{subfigure}
\caption{Comparison of sparse-grid (SG) and full-grid (FG) truncations for He (top) and H$_2$ (bottom), 
showing the maximum energy deviation $\max_{t\in [0,T]} \Delta E^{\mathrm{type}}(t)$ 
and projection error $\max_{t\in [0,T]} \mathcal{E}^{\mathrm{type}}(t)$. 
The results demonstrate that SG achieves significantly smaller errors than FG 
for the same dimension, thereby reducing computational cost while maintaining accuracy.}
  \label{fig:sg-vs-fg}
\end{figure}

\section{Preliminaries}\label{sec:3}
This section introduces several fundamental analytic tools, 
such as Strichartz estimates, Hardy inequalities, and Sobolev's inequalities, 
which play a key role in the proofs of the existence, mixed regularity, 
and approximation results established later in the paper. 

\subsection{Strichartz estimates}
To study the $N$-body problem \eqref{eq:1.1} in the functional space $X_{p,T}$,
we employ Strichartz estimates from \cite{Yajima2}.
These estimates are crucial for proving our existence result and also play a central role in the approximation analysis,
particularly in the proofs of Proposition~\ref{th:2.8} and Corollary~\ref{cor:2.5}.

Before going further, we recall the free propagator 
$U_0(t)=\exp\!\left(\tfrac{i}{2}t\sum_{j=1}^N \Delta_j\right)$, and denote the integral operator $S$ and $Q$ respectively by
\begin{align}\label{eq:S}
    Su(t)&:=\int_0^t U_0(t-\tau)u(\tau)\,d\tau,\\[4pt]
\label{eq:Q}
    Qu(t)&:=\sum_{j=1}^N S\!\Big(V(\cdot,x_j)u(\cdot,x)\Big)(t)
         +\sum_{1\leq j<k\leq N} S\!\Big(W(x_j,x_k)u(\cdot,x)\Big)(t).
\end{align}
With this notation, Duhamel's formula shows that solutions $u$ of \eqref{eq:1.1} satisfies
\begin{equation}\label{duhamel}
    u(t) = U_0(t)u_0 - iQu(t).
\end{equation}

To establish the standard Strichartz estimate (Lemma~\ref{lem:2.4}), 
we first prove a dispersive estimate (Lemma~\ref{lem:dispersive}). 
Note that both results can be found in \cite[Lemmas~2.2 and~2.3]{Yajima2}, 
but here our versions are global-in-time. 
This is because in \cite{Yajima2}, an additional time-dependent magnetic potential $A(t,x)$ is added, which complicates the arguments.

\begin{lemma}[Dispersive estimate]\label{lem:dispersive}
Let $D\subset \{1,\cdots,N\}$ with $1\leq|D|\leq2$. Then
\[
\|U_0(t)g\|_{L^{\infty,2}_D}\lesssim |t|^{-3/2}\|g\|_{L^{1,2}_D}.
\]
\end{lemma}
\begin{proof}
For the case $D=\{j\}$ (i.e., $|D|=1$), the claim reduces to the standard dispersive estimate 
(see, e.g., \cite[Eq.~(2.22)]{tao}):
\begin{align*}
    \|U_0(t)g\|_{L^{\infty,2}_D}=\|e^{\frac{1}{2} it\sum_{k=1}^N \Delta_k}g\|_{L^{\infty,2}_j}= \|e^{\frac{1}{2}it \Delta_{j}}g\|_{L^{\infty,2}_D}\lesssim |t|^{-3/2}\|g\|_{L^{1,2}_D}.
\end{align*}
For the case $D=\{j,k\}$ (i.e., $|D|=2$), we use the change of variables 
$r_{j,k}=\tfrac{1}{2}(x_j-x_k)$ and $R_{j,k}=\tfrac{1}{2}(x_j+x_k)$.
Note that 
\begin{align}\label{eq:Rjk-nabla}
    \mathcal{R}_{j,k} \nabla_j= \frac{1}{2}(\nabla_{r_{j,k}}+\nabla_{R_{j,k}})\mathcal{R}_{j,k},\qquad \mathcal{R}_{j,k} \nabla_k= \frac{1}{2}(\nabla_{R_{j,k}}-\nabla_{r_{j,k}})\mathcal{R}_{j,k}
\end{align}
and
\begin{align}\label{eq:Rjk-delta}
    \mathcal{R}_{j,k}\Delta_j+\mathcal{R}_{j,k}\Delta_k=\frac12\Delta_{r_{j,k}}\mathcal{R}_{j,k}+\frac12\Delta_{R_{j,k}}\mathcal{R}_{j,k}.
\end{align}
Thus, 
\[
\mathcal{R}_{j,k}U_0(t)u=\widetilde{U}_0(t)\mathcal{R}_{j,k}u, \quad \text{with}\quad \widetilde{U}_0(t)=\exp\!\Big(\tfrac{i}{2}\Big(\sum_{m\neq j,k}\Delta_m
+\tfrac12\Delta_{r_{j,k}}+\tfrac12\Delta_{R_{j,k}}\Big)\Big).
\]
Therefore, 
\[
\begin{aligned}
\|U_0(t)g\|_{L^{\infty,2}_{j,k}}&=\|\mathcal{R}_{j,k}U_0(t)g\|_{L^{\infty,2}_{r_{j,k}}}=\|\widetilde{U}_0(t)\mathcal{R}_{j,k}g\|_{L^{\infty,2}_{r_{j,k}}}\lesssim |t|^{-3/2}\|\mathcal{R}_{j,k}g\|_{L^{1,2}_{r_{j,k}}}\leq |t|^{-3/2}\|g\|_{L^{1,2}_{j,k}}.
\end{aligned}
\]
Hence the lemma.
\end{proof}

Using the above dispersive estimates, we obtain the following Strichartz estimates. 
\begin{lemma}[Strichartz estimates]\label{lem:2.4}
For $D,D'\subset\{1,\cdots,N\}$ with $1\leq |D|,|D'|\leq 2$ and $2\leq p,\widetilde{p}\leq 6$, we have
\begin{subequations}\label{eq:optim}
\begin{align}
&\|U_0(t)g\|_{L_t^{\theta_p}(\R,L^{p,2}_{D})}\lesssim_p \|g\|_{\cH},\label{eq:2.2a}\\
&\Big\|\int_{\R} U_0(s)^*u(s)\,ds\Big\|_{\cH}\lesssim_{\widetilde{p}} \|u\|_{L_t^{\theta_{\widetilde{p}}'}(\R,L^{\widetilde{p}',2}_{D'})},\label{eq:2.2b}\\
&\|Su\|_{L_t^{\theta_p}(\R,L^{p,2}_{D})}\lesssim_{p,\widetilde{p}} \|u\|_{L_t^{\theta_{\widetilde{p}}'}(\R,L^{\widetilde{p}',2}_{D'})}.\label{eq:2.2c}
\end{align}
\end{subequations}
\end{lemma}
These are the standard Strichartz estimates. One can easily obtain these estimates by using \cite{Keel}.

\subsection{Hardy inequalities and Sobolev's inequalities}\label{sec:2}
To study mixed regularity, we need some Hardy-type inequalities established in 
\cite[Lemma~3.7 and Corollary~3.8]{meng}.
\begin{lemma}\cite[Lemma 3.7]{meng}\label{Hardy1}
Let $a\in \R^3$ and define
\[
Y_{s}(\R^3):=\Big\{f(y)\in L^2(\R^3)\;:\; |y-a|^{2s-2}\nabla_y f \in L^2(\R^3)\Big\}.
\]
Then, for $s\in[1,3/2)$ and $f\in Y_{s}(\R^3\times \R^3)$, we have
\[
\left\|\frac{f}{|\,\cdot-a|^{s}}\right\|_{L^2(\R^3)}
   \;\leq\; \frac{2}{|2s-3|}\left\|\frac{\nabla_y f}{|\,\cdot-a|^{s-1}}\right\|_{L^2(\R^3)}.
\]
\end{lemma}
and
\begin{lemma}\cite[Corollary~3.8]{meng}\label{Hardy}
Define the functional space $Y_{\mathrm{anti},s}(\R^3\times \R^3)$ by
\[
Y_{\mathrm{anti},s}(\R^3\times \R^3)
:=\Big\{f\in L^2(\R^3\times \R^3)\;:\; f(y,z)=-f(z,y),\;\;
|y-z|^{\,s-2}\,\nabla_y\otimes \nabla_z f \in L^2(\R^3\times \R^3)\Big\}.
\]
Then for $s\in [2,5/2)$ and $f\in Y_{\mathrm{anti},s}(\R^3\times \R^3)$, we have
\[
\left\|\frac{f}{|y-z|^{s}}\right\|_{L^2(\R^3\times \R^3)}
\;\leq\; \frac{4}{|2s-5|\,|2s-3|}
\left\|\frac{\nabla_y \otimes \nabla_z f}{|y-z|^{s-2}}\right\|_{L^2(\R^3\times \R^3)}.
\]
\end{lemma}
We also need some Sobolev-type inequalities on the functional spaces $L^{p,2}_D$.  
To this end, we recall the Mikhlin multiplier theorem in the setting of $L^{p,2}_j$.  
Let $a:\R^n\to\C$ be a bounded measurable function.  
It defines a bounded linear operator $T_a:L^2(\mathbb{R}^n,\mathbb{C})\rightarrow L^2(\mathbb{R}^n,\mathbb{C})$ acting on $u\in L^2(\R^n\times\R^m,\mathbb{C})$ by
\[
T_a u:=\widecheck{a\widehat{u}}
\]
Here $\widehat{u}(\xi_1,y_2):=\int_{\R^n}e^{-2\pi i y_1\cdot \xi_1}u(y_1,y_2)dy_1$ is the Fourier transform for $x\in \R^n$ and $\widecheck{u}(y_1,y_2)$ is the inverse.
\begin{theorem}[Mikhlin multiplier]\label{th:a.2}
Let $m,n\in \mathbb{N}_0$.  
Suppose $a:\R^n\setminus\{0\}\to\C$ is a $C^{n+2}$ function satisfying
\[
|\partial^\alpha a(\xi)|\leq C|\xi|^{-\alpha}
\]
for every $\xi\in\mathbb{R}^n\setminus\{0\}$ and every multi-index $\alpha=(\alpha_1,\cdots,\alpha_n)\in\mathbb{N}_0^n$ with $|\alpha|\leq n+2$. 
Then for any $1<p<\infty$,
\[
   \|T_a f\|_{L^p(\R^n,L^2(\R^m))} \;\lesssim_n\; \|f\|_{L^p(\R^n,L^2(\R^m))}.
\]
\end{theorem}
The proof of Theorem \ref{th:a.2} is exactly the same as the standard Mikhlin multiplier theorem (see, e.g., \cite[Theorem 8.2]{schlag}). The only difference is that we use the Calder\'on-Zygmund inequality in \cite[Theorem 2.1.9]{singular} instead of the normal one.

We now state Sobolev inequalities associated with the functional spaces $L^{p,2}_{j}$ and $L^{p,2}_{j,k}$.
\begin{theorem}\label{th:2.7}
Let $1<p<\infty$ and $D\subset\{1,\dots,N\}$ with $1\le |D|\le 2$. Then we have 
\begin{subequations}
\begin{align}
   \|\nabla_j g\|_{L^{p,2}_D} &\;\lesssim_p\; \|(1-\Delta_j)^{1/2}g\|_{L^{p,2}_D}, \label{eq:2.6a}\\
   \|g\|_{L^{p,2}_D} &\;\lesssim_p\; \|(1-\Delta_j)^{1/2}g\|_{L^{p,2}_D}. \label{eq:2.6b}
\end{align}
\end{subequations}
\end{theorem}
\begin{proof}
We first consider the case $j\not\in D$. By applying the Plancherel theorem on variable $x_j$, it is easy to see that \eqref{eq:2.6a}-\eqref{eq:2.6b} hold. Now suppose $D=\{j\}$. Estimates \eqref{eq:2.6a} and \eqref{eq:2.6b} follow directly from Theorem~\ref{th:a.2} with $n=3$ and multiplier
\[
a(\xi)=\xi(1+|\xi|^2)^{-1/2}\quad\mbox{or}\quad  a(\xi)=(1+|\xi|^2)^{-1/2}, \qquad \xi\in\R^3. 
\]

Finally, suppose $D=\{j,k\}$ with $k\in\{1,\cdots,N\}\setminus\{j\}$.  
We only prove for \eqref{eq:2.6a}. The estimate \eqref{eq:2.6b} follows analogously. Using the change of variables \eqref{Lp2-rjk}, 
the identities \eqref{eq:Rjk-nabla}--\eqref{eq:Rjk-delta}, and applying the Plancherel theorem in $R_{j,k}$, we obtain
\[
\begin{aligned}
 \|\nabla_jg\|_{L^{p,2}_{j,k}} &= \|\mathcal{R}_{j,k}\nabla_jg\|_{L^{p,2}_{j,k}} = \frac{1}{2}\|(\nabla_{r_{j,k}} + \nabla_{R_{j,k}})\mathcal{R}_{j,k}g\|_{L^{p,2}_{r_{j,k}}}\\
 &= \frac{1}{2}\|\mathcal{F}_{R_{j,k}} \left((\nabla_{r_{j,k}}+\nabla_{R_{j,k}})\mathcal{R}_{j,k}g\right)\|_{L^{p,2}_{r_{j,k}}} =\frac{1}{2}\|(\nabla_{r_{j,k}}+i2\pi\xi_{R_{j,k}})\mathcal{F}_{R_{j,k}}\mathcal{R}_{j,k}g\|_{L^{p,2}_{r_{j,k}}}\\
 & = \frac{1}{2}\|\exp(i2\pi r_{j,k} \cdot \xi_{R_{j,k}})(\nabla_{r_{j,k}}+i2\pi\xi_{R_{j,k}})\mathcal{F}_{R_{j,k}}\mathcal{R}_{j,k}g\|_{L^{p,2}_{r_{j,k}}} = \frac{1}{2}\|\nabla_{r_{j,k}}\exp{(i2\pi r_{j,k}\cdot\xi_{R_{j,k}})}\mathcal{F}_{R_{j,k}}\mathcal{R}_{j,k}g\|_{L^{p,2}_{r_{j,k}}}\\
   &\lesssim_p\|(1-\frac{1}{4}\Delta_{r_{j,k}})^{1/2}\exp{(i2\pi r_{j,k}\cdot\xi_{R_{j,k}})}\mathcal{F}_{R_{j,k}}\mathcal{R}_{j,k}g\|_{L^{p,2}_{r_{j,k}}}=\|(1-\frac{1}{4}|\nabla_{r_{j,k}}+i2\pi \xi_{R_{j,k}}|^2)^{1/2}\mathcal{F}_{R_{j,k}}\mathcal{R}_{j,k}g\|_{L^{p,2}_{r_{j,k}}}\\
   &= \|(1-\frac{1}{4}|\nabla_{r_{j,k}}+\nabla_{R_{j,k}}|^2)^{1/2}\mathcal{R}_{j,k}g\|_{L^{p,2}_{r_{j,k}}}= \|(1-\Delta_j)^{1/2} g\|_{L^{p,2}_{j,k}}.
\end{aligned}
\]
Here we use the fact that for any function $f(y)$ with $y\in \R^3$, 
\[
\begin{aligned}
 (1-\frac{1}{4}\Delta_{y})^{1/2}\exp{(i2\pi a\cdot \xi_y)} f(y) &= \mathcal{F}_y^{-1} \Big[(1-\frac{1}{4} \cdot (i2\pi \xi_y)^2)^{1/2}\mathcal{F}_y\big(\exp{(i2\pi a\cdot \xi_y)} f\big)(\xi_y)\Big](y)\\
&=\mathcal{F}_y^{-1} \Big[(1+\pi^2|\xi_y|^2)^{1/2}\mathcal{F}_y\big(f\big)(\xi_y+a)\Big](y)= (1-\frac{1}{4}|\nabla_y+i2\pi a|^2)^{1/2}f(y)
\end{aligned}
\]
and for the last identity, we use \eqref{eq:Rjk-delta}. 
\end{proof}
\subsection{Approximation estimates}
\label{sec:approx-estimate}
Normally, operators bounded on $\cH$ are not bounded on $L^{p,2}_{D}$. However, the next result shows that if a bounded operator $\mathcal{P}$ on $\cH$ commutes with the free propagator $U_0$, i.e., $[\mathcal{P},U_0]=0$, 
then after composition with the operator $S$ it is also bounded on $L^{p,2}_{D}$. 
This property is a key ingredient in the analysis of our sparse grid approximation. 
\begin{proposition}\label{th:2.8}
Let $2\leq p,{\widetilde{p}}<6$. Suppose $\mathcal{P}$ is a bounded operator on $\cH$ such that $[\mathcal{P},U_0]=0$ and $\|\mathcal{P}\|_{\mathcal{B}(\cH)}\leq 1$. 
Then we have 
\begin{align}\label{eq:PS-strichartz}
    \|\mathcal{P}Su(\cdot,x)\|_{L_t^{\theta_p}(\R,L^{p,2}_{D})}
   \;\lesssim_{p,{\widetilde{p}}}\;
   \|u\|_{L_t^{\theta_{\widetilde{p}}'}(\R,L^{{\widetilde{p}}',2}_{D'})}.
\end{align}
\end{proposition}

\begin{proof}
We're going to use the Christ--Kiselev lemma~\cite{ChristKiselev} to prove this lemma. Since $\mathcal{P}$ and $U_0$ commute, we have
\[
\left\|\mathcal{P}\int_\R U_0(t-s)u(s,x)ds\right\|_{L_t^{\theta_p}(\R,L^{p,2}_{D})}=\left\|U_0(t)\mathcal{P}\int_\R U_0(s)^*u(s,x)ds\right\|_{L_t^{\theta_p}(\R,L^{p,2}_{D})}.
\]
Applying \eqref{eq:2.2a} and using $[\mathcal{P},U_0]=0$ gives 
\[
\left\|\mathcal{P}\int_\R U_0(t-s)u(s,x)ds\right\|_{L_t^{\theta_p}(\R,L^{p,2}_{D})}\lesssim_{p}\left\|\mathcal{P}\int_\R U_0(s)^*u(s,x)ds\right\|_{\cH}.
\]
Since $\|\mathcal{P}\|_{\mathcal{B}(\cH)}\leq 1$, by \eqref{eq:2.2b} we further obtain
\[
\left\|\mathcal{P}\int_\R U_0(t-s)u(s,x)ds\right\|_{L_t^{\theta_p}(\R, L^{p,2}_{D})}\lesssim_p\left\|\int_\R U_0(s)^*u(s,x)ds\right\|_{\cH}\lesssim_{p,{\widetilde{p}}}\|u\|_{L_t^{\theta_{\widetilde{p}}'}(\R,L^{{\widetilde{p}}',2}_{D'})}.
\]
Then \eqref{eq:PS-strichartz} follows from the Christ--Kiselev lemma, hence the proposition.
\end{proof}
In addition, if the projector $\mathcal{P}_R$ further satisfies Assumption~\ref{ass:1.8}, 
we obtain the following estimate~\eqref{eq:2.3b} as a corollary of Proposition~\ref{th:2.8}. 
\begin{corollary}\label{cor:2.5}
Let $\mathcal{P}_R$ satisfy Assumption~\ref{ass:1.8}. 
For $D,D'\subset\{1,\cdots,N\}$ with $1\leq|D|,|D'|\leq 2$ and $2\leq p,\widetilde{p}<6$, we have
\begin{subequations}
\begin{align}
    \|\mathcal{P}_{R} S u\|_{L_t^{\theta_p}(\R,L^{p,2}_{D})}
    &\lesssim_{p,\widetilde{p}} 
      \|u\|_{L_t^{\theta_{\widetilde{p}}'}(\R,L^{\widetilde{p}',2}_{D'})},
      \label{eq:2.3a}\\[0.5em]
    \left\|
      \Big(\sum_{1\leq \ell\leq 2}\cL_{I_\ell}\Big)^{-1} 
      (1-\mathcal{P}_{R}) S u
    \right\|_{L_t^{\theta_p}(\R,L^{p,2}_{D})}
    &\lesssim_{p,\widetilde{p}} 
      \frac{1}{R}\,
      \|u\|_{L_t^{\theta_{\widetilde{p}}'}(\R,L^{\widetilde{p}',2}_{D'})}.
      \label{eq:2.3b}
\end{align}
\end{subequations}
\end{corollary}
\begin{proof}
By the definition of $\mathcal{P}_{R}$, we have
\[
[\mathcal{P}_{R},U_0]=0, 
\qquad \|\mathcal{P}_{R} u\|_{\cH}\leq \|u\|_{\cH}.
\]
Thus, Proposition~\ref{th:2.8} with $\mathcal{P}=\mathcal{P}_R$ gives \eqref{eq:2.3a}.  For \eqref{eq:2.3b}, recall from Assumption~\ref{ass:1.8} that
\begin{align}\label{eq:3.5}
   R\Big\|\Big(\sum_{\ell=1,2}\cL_{I_\ell}\Big)^{-1}(1-\mathcal{P}_R)u\Big\|_\cH
   \;\leq\;\|u\|_\cH.
\end{align}
Since $[\cL_{I_\ell},U_0]=0$, we may set 
$\mathcal{P}:=R\,(1-\mathcal{P}_R)\Big(\sum_{\ell=1,2}\cL_{I_\ell}\Big)^{-1}$. 
Then $[\mathcal{P},U_0]=0$ and $\|\mathcal{P}\|_{\cB(\cH)}\leq 1$. Applying Proposition~\ref{th:2.8} to this $\mathcal{P}$ yields \eqref{eq:2.3b}.
\end{proof}

\section{Existence and mixed regularity of solutions}\label{sec:4}
In this section, we study the existence and mixed regularity of solutions to \eqref{eq:1.1}. 
The proof of existence and uniqueness follows essentially from \cite{Yajima2} with minor adjustments to fit our setting. 
Our main focus here is on establishing the mixed regularity of the solution.

\subsection{Existence of solutions}\label{sec:exist}
In this subsection, we study the existence of solutions to \eqref{eq:1.1} in $X_{p,T}$. 
We present the proof of existence using the same method as for mixed regularity, 
serving as a technical preparation for the next theorem. 

\begin{proof}[Proof of Theorem \ref{th:existence}]
The key step is to show 
$\|Qu\|_{X_{p,T}}\leq C T^{\theta}\|u\|_{X_{p,T}}$
for some $\theta>0$, $2<p\leq 6$, and $T>0$.  
To this end, we consider the scalar product version:  
for any $D\subset\{1,\cdots,N\}$ with $1\leq |D|\leq 2$, 
$u(t,x)\in X_{p,T}$, and $v(t,x)\in L_t^{\theta'_{p}}(I_T,L^{p',2}_{D})$ or 
$v(t,x)\in L_t^{1}(I_T,\cH)$, we are going to prove
\begin{align*}
   \left|\int_{0}^T\langle (Qu)(t),v(t) \rangle \,dt\right|
   \;\lesssim\; T^{\theta}\|u\|_{X_{p,T}}
   \min\big\{\|v\|_{L_t^{\theta'_{p}}(I_T,L^{p',2}_{D})},\,\|v\|_{L^1_t(I_T,\cH)}\big\}, 
\end{align*}
To do so, we split the analysis into two parts: the electron–nucleus potentials and the electron–electron potentials.

\medskip

{\bf Step 1. Study of the potentials between electrons and nuclei.} For any $0< \alpha< \frac{1}{2}$, we have
\begin{align}\label{eq:SVu1}
   \MoveEqLeft \left|\int_0^T\Big\langle S\Big(\frac{1}{|x_j-a_\mu(\cdot)|}u(\cdot,x)\Big)(t),v(t,x) \Big\rangle \;dt\right|\notag\\
    &= \left|\int_0^T\Big\langle \frac{1}{|x_j-a_\mu(s)|^\alpha}u(s,x),\frac{1}{|x_j-a_\mu(s)|^{1-\alpha}} S^*\big(v(\cdot,x)\big)(s) \Big\rangle \;ds\right|\notag\\
    &\leq \int_0^T \left\|\frac{1}{|x_j-a_\mu(s)|^\alpha}u(s,x)\right\|_{\cH}\left\|\frac{1}{|x_j-a_\mu(s)|^{1-\alpha}} S^*\big(v(\cdot,x)\big)(s)\right\|_{\cH}ds.
\end{align}
According to the H\"older inequality, for any $2\leq r< \frac{3}{\alpha}$ and $\frac{1}{r}+\frac{1}{p_1}=\frac{1}{2}$,
\begin{align}\label{eq:u-p1}
 \MoveEqLeft  \Big\| \frac{1}{|x_j-a_\mu(s)|^\alpha}u(s,x)\Big\|_{\cH}= \Big\| \frac{1}{|x_j-a_\mu(s)|^\alpha} \|u(s,x_j,\cdot)\|_{L^2((\R^3)^{N-1})}\Big\|_{L^2_j(\R^3)}\notag\\
    &\leq \|u(s)\|_{\cH}+\left\|\frac{1}{|x_j-a_\mu(s)|^\alpha}\mathbbm{1}_{|x_j-a_\mu(s)|\leq 1}\right\|_{L^r_j(\R^3)}\|u(s)\|_{L^{p_1,2}_j}\lesssim_{\alpha,p_1} \|u(s)\|_{\cH}+\|u(s)\|_{L^{p_1,2}_j}.
\end{align}
Similarly, for any $2\leq \widetilde{r}< \frac{3}{1-\alpha}$ and $\frac{1}{\widetilde{r}}+\frac{1}{\widetilde{p}}=\frac{1}{2}$,
\begin{align}\label{eq:Sv}
    \Big\|\frac{1}{|x_j-a_\mu(s)|^{1-\alpha}} (S^*v)(s) \Big\|_{\cH}\lesssim_{\alpha,\widetilde{p}} \|(S^*v)(s)\|_{\cH}+\|(S^*v)(s)\|_{L^{\widetilde{p},2}_j}.
\end{align}
Next we apply the dual form of the Strichartz estimate \eqref{eq:2.2c} to $S^*v$.  
To do so, we have to add the restriction $2\leq \widetilde{p}\leq 6$ to use the Strichartz estimates. Combining with $2\leq \widetilde{r}< \tfrac{3}{1-\alpha}$ and $\tfrac{1}{\widetilde{r}}+\tfrac{1}{\widetilde{p}}=\tfrac{1}{2}$, we have
\begin{align}\label{eq:tildep1}
    \tfrac{6}{1+2\alpha}<\widetilde{p}\leq 6.
\end{align}
Therefore, for any $D\subset \{1,\cdots,N\}$ with $0\leq |D|\leq 2$, and any $\widetilde{p}$ as in \eqref{eq:tildep1}, $p$ with $2<p\leq 6$, we have
\begin{align}\label{eq:Sv1}
    \|(S^*v)(s)\|_{L_t^\infty(I_T,\cH)} 
      &\lesssim_{p,\widetilde{p}} 
      \min\big\{\|v\|_{L_t^{\theta'_{p}}(I_T, L^{p',2}_D)},\,
                \|v\|_{L^1_t(I_T,\cH)}\big\},\\[0.5em]
\label{eq:Sv2}
    \|(S^*v)(s)\|_{L_t^{\theta_{\widetilde{p}}}(I_T,L^{\widetilde{p},2}_j)}
      &\lesssim_{p,\widetilde{p}} 
      \min\big\{\|v\|_{L_t^{\theta'_{p}}(I_T, L^{p',2}_D)},\,
                \|v\|_{L^1_t(I_T,\cH)}\big\}.
\end{align}
Finally, to ensure that $Q$ maps $X_{p,T}$ into itself, we set $p_1=p$ with $2\leq p\leq 6$.  
Then, from $2\leq r< \tfrac{3}{\alpha}$ and $\tfrac{1}{r}+\tfrac{1}{p_1}=\tfrac{1}{2}$,
\begin{align}\label{eq:p1}
   \tfrac{6}{3-2\alpha}< p\leq 6.
\end{align}

As a result, from \eqref{eq:SVu1}--\eqref{eq:Sv} and \eqref{eq:Sv1}--\eqref{eq:Sv2} we deduce that, under conditions \eqref{ass1}--\eqref{ass2} in Assumption~\ref{ass} on $\alpha,p,\widetilde{p}$, for any $0<T<1$ and any $D\subset\{1,\cdots,N\}$ with $1\leq |D|\leq 2$,
\begin{align}\label{eq:SV-scalar}
   \MoveEqLeft \left|\int_0^T\Big\langle S\Big(\frac{1}{|x_j-a_\mu(\cdot)|}u(\cdot,x)\Big)(t),v(t,x) \Big\rangle \;dt\right|\notag\\
   &\leq \int_0^T \left\|\frac{1}{|x_j-a_\mu(s)|^\alpha}u(s,x)\right\|_{\cH}\left\|\frac{1}{|x_j-a_\mu(s)|^{1-\alpha}} S^*\big(v(\cdot,x)\big)(s)\right\|_{\cH}ds\notag\\
   &\lesssim_{\alpha, p,\widetilde{p}} \Big(\|u\|_{L_t^1(I_T,\cH)}+ \|u(s)\|_{L_t^1(I_T,L^{p,2}_j)}\Big)\|S^*v\|_{L_t^\infty(I_T,\cH)}+\Big(\|u\|_{L_t^{\theta'_{\widetilde{p}}}(I_T,\cH)}+ \|u\|_{L_t^{\theta'_{\widetilde{p}}}(I_T,L^{p,2}_j)}\Big)\|S^*v\|_{L_t^{\theta_{\widetilde{p}}}(I_T,L^{\widetilde{p},2}_j)}\notag\\
   &\lesssim_{\alpha,p,\widetilde{p}} T^{1/\theta'_{\widetilde{p}}-1/\theta_{p}}\Big(\|u\|_{L_t^{\infty}(I_T,\cH)}+ \|u\|_{L_t^{\theta_{p}}(I_T,L^{p,2}_j)}\Big)\min\{\|v\|_{L_t^{\theta'_{p}}(I_T, L^{p',2}_D)},\|v\|_{L^1_t(I_T,\cH)}\}.
\end{align}
Here we also use the assumption $1/\theta'_{\widetilde{p}}-1/\theta_{p}>0$.  
Since $u\in X_{p,T}$, by duality we obtain, for any $0<T<1$ and any $D\subset\{1,\cdots,N\}$ with $1\leq |D|\leq 2$,
\[
   \max\!\left\{\Big\|S\Big(\tfrac{1}{|x_j-a_\mu(\cdot)|}u(\cdot,x)\Big)\Big\|_{L_t^{\infty}(I_T,\cH)},\,
              \Big\|S\Big(\tfrac{1}{|x_j-a_\mu(\cdot)|}u(\cdot,x)\Big)\Big\|_{L_t^{\theta_{p}}(I_T,L^{p,2}_D)}\right\}
   \lesssim_{\alpha,p,\widetilde{p}} T^{1/\theta'_{\widetilde{p}}-1/\theta_{p}} \|u\|_{X_{p,T}}.
\]
Consequently, we infer that under under condition \eqref{ass1}-\eqref{ass2} in Assumption \ref{ass} on $\alpha,p,\widetilde{p}$ and $0<T<1$,
\begin{align}\label{eq:SV}
    \left\|S(V(\cdot,x_j)u)\right\|_{X_{p,T}}\lesssim_{\alpha,p_,\widetilde{p}} ZT^{1/\theta'_{\widetilde{p}}-1/\theta_{p}} \|u\|_{X_{p,T}}
\end{align}
\medskip

{\bf Step 2. Study of the potentials between electrons and electrons.} The proof of Step 2 is essentially the same as for Step~1. So we only highlight the differences.

For $\frac{1}{|x_j-x_k|}$, the estimates \eqref{eq:u-p1} and \eqref{eq:Sv} become: For any $2\leq r< \frac{3}{\alpha}$ with $\frac{1}{r}+\frac{1}{p_1}=\frac{1}{2}$,
\begin{align*}
    \left\|\tfrac{1}{|x_j-x_k|^{\alpha}}u(s)\right\|_{\cH}
    \;\lesssim_{\alpha,p}\; \|u(s)\|_{\cH}+\|u(s)\|_{L^{p,2}_{j,k}},
\end{align*}
and for any $2\leq \widetilde{r}< \tfrac{3}{1-\alpha}$ with $\tfrac{1}{\widetilde{r}}+\tfrac{1}{\widetilde{p}}=\tfrac{1}{2}$,
\begin{align*}
    \left\|\tfrac{1}{|x_j-x_k|^{1-\alpha}}(S^*v)(s)\right\|_{\cH}
    \;\lesssim_{\alpha,\widetilde{p}}\; \|(S^*v)(s)\|_{\cH}+\|(S^*v)(s)\|_{L^{\widetilde{p},2}_{j,k}}.
\end{align*}
Proceeding as in Step~1, we need to set $p_1=p$, and $p$ satisfies \eqref{eq:p1}. Then under condition \eqref{ass1}-\eqref{ass2} in Assumption~\ref{ass} on $\alpha,p,\widetilde{p}$, for any $0<T<1$ and any $D\subset\{1,\cdots,N\}$ with $1\leq |D|\leq 2$, we obtain
\begin{align}\label{eq:SW-scalar}
   \MoveEqLeft \left|\int_0^T\Big\langle S\Big(\frac{1}{|x_j-x_k|}u(\cdot,x)\Big)(t),v(t,x) \Big\rangle dt\right|\leq \int_0^T\left\| \frac{1}{|x_j-x_k|^\alpha} u(s,x)\right\|_{\cH}\left\|\frac{1}{|x_j-x_k|^{1-\alpha}}(S^*v(\cdot,x))(s) \right\|_{\cH} ds \notag\\
&\qquad\qquad\qquad\qquad\lesssim_{\alpha,p,\widetilde{p}} T^{1/\theta'_{\widetilde{p}}-1/\theta_{p}}\Big(\|u\|_{L_t^{\infty}(I_T,\cH)}+ \|u\|_{L_t^{\theta_{p}}(I_T,L^{p,2}_{j,k})}\Big)\min\{\|v\|_{L_t^{\theta'_{p}}(I_T, L^{p',2}_D)},\|v\|_{L^1_t(I_T,\cH)}\}.
\end{align}
As a result, under condition \eqref{ass1}-\eqref{ass2} in Assumption \ref{ass} on $\alpha,p,\widetilde{p}$ and $0<T<1$,
\begin{align}\label{eq:SWu1}
    \left\|S(W(x_j,x_k)u)\right\|_{X_{p,T}}\lesssim_{\alpha,p,\widetilde{p}} T^{1/\theta'_{\widetilde{p}}-1/\theta_{p}} \|u\|_{X_{p,T}}.
\end{align}

\medskip

{\bf Step 3. Conclusion.} From \eqref{eq:SV} and \eqref{eq:SWu1}, we infer that under condition \eqref{ass1}-\eqref{ass2} in Assumption \ref{ass} on $\alpha,p,\widetilde{p}$ and $0<T<1$, there exists a constant $C_{T,1}:=C_{T,1}(\alpha,p,\widetilde{p})\geq 1$ such that
\begin{align}\label{eq:CT1}
    \left\|Qu\right\|_{X_{p,T}}\leq C_{T,1} (Z+N)NT^{1/\theta'_{\widetilde{p}}-1/\theta_{p}} \|u\|_{X_{p,T}}.
\end{align}
Now let $C_{T,1}(Z+N)NT^{1/\theta'_{\widetilde{p}}-1/\theta_{p}}\leq \frac{1}{2}$, then we have $\|Qu\|_{X_{p,T}}\leq \tfrac{1}{2}\|u\|_{X_{p,T}}$. Thus under Assumption \ref{ass} on $p,T$, we know $1+iQ$ is invertible on $X_{p,T}$. As a result,
\begin{align}\label{u-solution}
    u=(1+iQ)^{-1} (U_0(\cdot) u_0)
\end{align}
and
\begin{align*}
    \|u\|_{X_{p,T}}=\|(1+iQ)^{-1}(U_0(\cdot)u_0)\|_{X_{p,T}}\leq 2\|U_0(\cdot)u_0\|_{X_{p,T}}\lesssim_{p} \|u_0\|_{\cH}.
\end{align*}
This gives \eqref{eq:u-solution} and shows the uniqueness of the solution $u\in X_{p,T}$. By conservation law, $\|u\|_{ \cH}(t)=\|u_0\|_{\cH}$ for $t\in \R$. The standard continuation procedure for the solutions to \eqref{eq:1.1} yields a unique global-in-time solution $u\in X_{p,\infty}$. This completes the proof of Theorem \ref{th:existence}.
\end{proof}

\subsection{Mixed regularity}\label{sec:mix}
Now we can study the mixed regularity of the unique solution. In the following, we use
\begin{align}\label{frac-laplace-dec}
    (1-\Delta_j)^{1/2}=\frac{1}{(1-\Delta_j)^{1/2}}-\frac{\nabla_j}{(1-\Delta_j)^{1/2}}\cdot (\nabla_j),
\end{align}
and define
\begin{align}\label{L-jk}
    \cL_{I,j}=\prod_{m\in I\setminus\{j\}}(1-\Delta_m)^{1/2},\qquad \cL_{I,j,k}=\prod_{m\in I\setminus\{j,k\}}(1-\Delta_m)^{1/2}.
\end{align}
\begin{proof}[Proof of Theorem \ref{th:mixregularity}]
Before going further, we first show that $u(t,x)$ is antisymmetric with respect to $I$ for any $t\in I_T$ under condition \eqref{ass3} in Assumption \ref{ass} on $T$. 
Let $j,k\in I$. Since $u_0=-P_{j,k}u_0$ for the permutation operator $P_{j,k}$ defined in Definition~\ref{def:anti}, 
it follows from \eqref{u-solution} that
\begin{align*}
    u=-(1+iQ)^{-1}(U_0(\cdot)P_{j,k}u_0)
     =-P_{j,k}(1+iQ)^{-1}(U_0(\cdot)u_0)
     =-P_{j,k} u. 
\end{align*}
Here we use the fact that $U_0(\cdot) P_{j,k}=P_{j,k} U_0$ and $Q P_{j,k}=P_{j,k} Q$. 
Therefore, under Assumption \ref{ass} on $p,T$,
\begin{align}\label{u-anti}
    u=-P_{j,k} u \qquad \text{in}\quad X_{p,T},
\end{align}
which shows that $u$ is antisymmetric w.r.t, $I$ for any $t\in I_T$.

We now show that for any $D\subset\{1,\cdots,N\}$ with $1\leq |D|\leq 2$, 
and for any $u(t,x)\in X^1_{I,p,T}$ and $v(t,x)\in L_t^{\theta'_{p}}(I_T,L^{p',2}_{D})$ 
or $v(t,x)\in L_t^{\infty}(I_T,\cH)$, one has
\begin{align*}
   \left|\int_{0}^T\langle (\cL_I Qu)(t),v(t) \rangle \;dt\right|
   \;\lesssim\; T^\theta \|u\|_{X^1_{I,p,T}}
   \min\Big\{\|v\|_{L_t^{\theta'_{p}}(I_T,L^{p',2}_{D})},\;
               \|v\|_{L^\infty_0(I_T,\cH)}\Big\}
\end{align*}
for some $\theta>0$ and $2<p\leq 6$.  
As before, we split the analysis into the case of electron–nucleus potentials and the case of electron–electron potentials.

\medskip

{\bf Step 1. Study of the potentials between electrons and nuclei.} First of all, we assume $j\not\in I$. Then we have
\begin{align*}
    \cL_I \frac{1}{|x_j-a_\mu(t)|}u&= \frac{1}{|x_j-a_\mu(t)|}\cL_Iu. 
\end{align*}
According to \eqref{eq:SV}, we infer that under condition \eqref{ass1}-\eqref{ass2} in Assumption \ref{ass} and for any $0<T<1$,
\begin{align}\label{eq:LI-SV}
\left\|S\cL_I(V(\cdot,x_j)u)\right\|_{X^1_{I,p,T}}=\left\|S(V(\cdot,x_j)\cL_Iu)\right\|_{X^1_{I,p,T}}\lesssim_{\alpha,p_,\widetilde{p}} ZT^{1/\theta'_{\widetilde{p}}-1/\theta_{p}} \|u\|_{X^1_{I,p,T}}.
\end{align}

\medskip

Now we consider the case $j\in I$. By \eqref{frac-laplace-dec} we have
\begin{align}\label{eq:LI-V}
   \MoveEqLeft \cL_I \frac{1}{|x_j-a_\mu(t)|}u= (1-\Delta_j)^{-1/2}\left[\frac{1}{|x_j-a_\mu(t)|}\cL_{I,j}u\right]-\big[(1-\Delta_j)^{-1/2}\nabla_j\big]\cdot \nabla_j\left[\frac{1}{|x_j-a_\mu(t)|}\cL_{I,j}u\right]
\end{align}
where $\cL_{I,j}$ is defined by \eqref{L-jk}. For the first term on the right-hand side of \eqref{eq:LI-V}, from \eqref{eq:SV-scalar}, we infer that under condition \eqref{ass1}-\eqref{ass2} in Assumption \ref{ass} and for any $0<T<1$ and any $D\subset\{1,\cdots,N\}$ with $1\leq |D|\leq 2$,
\begin{align}\label{eq:LI-SV-scalar1}
\MoveEqLeft \left|\int_0^T\Big\langle 
S(1-\Delta_j)^{-1/2}\Big(\tfrac{1}{|x_j-a_\mu(\cdot)|}\cL_{I,j}u(\cdot,x)\Big),
v(t,x) \Big\rangle \;dt\right|\notag\\
&=\left|\int_0^T\Big\langle 
S\Big(\tfrac{1}{|x_j-a_\mu(\cdot)|}\cL_{I,j}u(\cdot,x)\Big),
(1-\Delta_j)^{-1/2}v(t,x) \Big\rangle \;dt\right|\notag\\
&\lesssim_{\alpha,p,\widetilde{p}} T^{1/\theta'_{\widetilde{p}}-1/\theta_{p}}
\Big(\|u\|_{L_t^{\infty}(I_T,\cH)}+ \|u\|_{L_t^{\theta_{p}}(I_T,L^{p,2}_j)}\Big)\notag\\
&\quad\times \min\{\|(1-\Delta_j)^{-1/2}v\|_{L_t^{\theta'_{p}}(I_T, L^{p',2}_{D})},
\|(1-\Delta_j)^{-1/2}v\|_{L^\infty_t(I_T, \cH)}\}\notag\\
&\lesssim_{\alpha,p,\widetilde{p}} T^{1/\theta'_{\widetilde{p}}-1/\theta_{p}}
\Big(\|u\|_{L_t^{\infty}(I_T,\cH)}+ \|u\|_{L_t^{\theta_{p}}(I_T,L^{p,2}_j)}\Big)
\min\{\|v\|_{L_t^{\theta'_{p}}(I_T, L^{p',2}_{D})},\|v\|_{L^\infty_t(I_T, \cH)}\}.
\end{align}
Here we use the fact that $[S,(1-\Delta_j)^{-1/2}]=0$ in the first equation and Theorem \ref{th:2.7} in the last inequality. 

\medskip
For the second term on the right-hand side of \eqref{eq:LI-V}, notice that
\begin{align*}
    \nabla_j\Big[\frac{1}{|x_j-a_\mu(\cdot)|}\, \cL_{I,j}u(\cdot,x)\Big]
    =\frac{1}{|x_j-a_\mu(\cdot)|}\, \cL_{I,j}\nabla_j u(\cdot,x)
    +\Big[\nabla_j\frac{1}{|x_j-a_\mu(\cdot)|}\Big] \cL_{I,j}u(\cdot,x),
\end{align*}
and by Lemma \ref{Hardy1},
\begin{align*}
 \MoveEqLeft   \left\||x_j-a_\mu(\cdot)|^{1-\alpha}\Big[\nabla_j\frac{1}{|x_j-a_\mu(\cdot)|}\Big] \cL_{I,j}u(\cdot,x)\right\|_{\cH}\lesssim \left\|\frac{1}{|x_j-a_\mu(\cdot)|^{1+\alpha}} \cL_{I,j}u(\cdot,x)\right\|_{\cH}\lesssim_\alpha \left\|\frac{1}{|x_j-a_\mu(\cdot)|^{\alpha}} \cL_{I,j}\nabla_j u(\cdot,x)\right\|_{\cH}.
\end{align*}
Thus,
\begin{align}\label{eq:4.20}
    \left\||x_j-a_\mu(\cdot)|^{1-\alpha}\nabla_j\Big[\frac{1}{|x_j-a_\mu(\cdot)|} \cL_{I,j}u(\cdot,x)\Big]\right\|_{\cH}\lesssim_\alpha \left\|\frac{1}{|x_j-a_{\mu}(\cdot)|^\alpha}\cL_{I,j}\nabla_j u\right\|_{\cH}.
\end{align}
Then we have
\begin{align}\label{eq:LI-SVu2}
   \MoveEqLeft \left|\int_0^T\Big\langle S[(1-\Delta_j)^{-1/2}\nabla_j]\cdot \left(\nabla_j\Big[\frac{1}{|x_j-a_\mu(\cdot)|} \cL_{I,j}u(\cdot,x)\Big]\right),\,v(t,x) \Big\rangle \;dt\right|\notag\\
   &\lesssim_{\alpha} \int_0^T\left\|\frac{1}{|x_j-a_\mu(s)|^\alpha}\,\cL_{I,j} \nabla_j u(s,x)\right\|_{\cH}
      \left\|\frac{1}{|x_j-a_\mu(s)|^{1-\alpha}}\, S^*\!\big((1-\Delta_j)^{-1/2}\nabla_jv(\cdot,x)\big)(s) \right\|_{\cH} \,ds.
\end{align}
Proceeding as for \eqref{eq:SV-scalar} and \eqref{eq:LI-SV-scalar1}, we infer that, under
conditions \eqref{ass1}–\eqref{ass2} in Assumption \ref{ass}, for any $0<T<1$ and any
$D\subset\{1,\cdots,N\}$ with $1\leq |D|\leq 2$, 
\begin{align}\label{eq:LI-SV-scalar2}
\MoveEqLeft \left|\int_0^T\Big\langle S[(1-\Delta_j)^{-1/2}\nabla_j]\cdot \left(\nabla_j\Big[\frac{1}{|x_j-a_\mu(\cdot)|} \cL_{I,j}u(\cdot,x)\Big]\right),v(t,x) \Big\rangle \;dt\right|\notag\\
&\lesssim_{\alpha,p,\widetilde{p}} T^{1/\theta'_{\widetilde{p}}-1/\theta_{p}}\Big(\|\nabla_j\cL_{I,j} u\|_{L_t^{\infty}(I_T,\cH)}+ \|\nabla_j\cL_{I,j} u\|_{L_t^{\theta_{p}}(I_T,L^{p,2}_j)}\Big)\notag\\
&\quad\times \min\big\{\|(1-\Delta_j)^{-1/2}\nabla_j v\|_{L_t^{\theta'_{p}}(I_T, L^{p',2}_{D})}, \|(1-\Delta_j)^{-1/2}\nabla_jv\|_{L^1_t(I_T,\cH)}\big\}\notag\\
&\lesssim_{\alpha,p,\widetilde{p}} T^{1/\theta'_{\widetilde{p}}-1/\theta_{p}}\Big(\|\cL_{I} u\|_{L_t^{\infty}(I_T,\cH)}+ \|\cL_{I} u\|_{L_t^{\theta_{p}}(I_T,L^{p,2}_j)}\Big)\min\{\|v\|_{L_t^{\theta'_{p}}(I_T, L^{p',2}_D)},\|v\|_{L^1_t(I_T,\cH)}\}, 
\end{align}
where in the last inequality we used Theorem \ref{th:2.7}. Consequently, \eqref{eq:LI-V},
\eqref{eq:LI-SV-scalar1}, \eqref{eq:LI-SV-scalar2} (for $j\in I$) and \eqref{eq:LI-SV} (for $j\not\in I$) imply that, under conditions
\eqref{ass1}–\eqref{ass2} in Assumption \ref{ass} and for any $0<T<1$,
\begin{align}\label{eq:LI-SV-total}
    \left\|\cL_I S\big(V(\cdot,x_j)u\big)\right\|_{X^1_{I,p,T}}
\lesssim_{\alpha,p,\widetilde{p}} Z\,T^{1/\theta'_{\widetilde{p}}-1/\theta_{p}} \,\|u\|_{X^1_{I,p,T}}.
\end{align}

\medskip

{\bf Step 2. Study of the potentials between electrons and electrons.} 
We now consider electron–electron interaction terms of the form $W(x_j,x_k)$. 
The analysis is divided into three cases: $\{j,k\}\cap I=\emptyset$, $|\{j,k\}\cap I|=1$, and $\{j,k\}\subset I$.

{\bf Case 1.} If $\{j,k\}\cap I=\emptyset$, then
\[
    \cL_I\frac{1}{|x_j-x_k|}u \;=\; \frac{1}{|x_j-x_k|}\cL_Iu.
\]
Applying \eqref{eq:SWu1}, we deduce that under conditions \eqref{ass1}--\eqref{ass2} in Assumption~\ref{ass} and for $0<T<1$,
\begin{align}\label{case1}
    \big\|S\cL_I(|x_j-x_k|^{-1}u)\big\|_{X_{p,T}}
    = \big\|S(|x_j-x_k|^{-1}\cL_I u)\big\|_{X_{p,T}}
    \;\lesssim_{\alpha,p,\widetilde{p}} T^{1/\theta'_{\widetilde{p}}-1/\theta_{p}} \|u\|_{X^1_{I,p,T}}.
\end{align}

{\bf Case 2.}  For the case $|\{j,k\}\cap I|=1$, without loss of generality we assume $j\in I$ and $k\not\in I$ (the case $j\not\in I$, $k\in I$ is analogous).  
Then
\begin{align*}
    \cL_I \frac{1}{|x_j-x_k|}u=(1-\Delta_j)^{1/2}\frac{1}{|x_j-x_k|}\cL_{I,j}u.
\end{align*}
The argument is essentially the same as in Step~1, with $a_\mu(t)$ replaced by $x_k$.  From \eqref{eq:LI-SVu2}, we obtain
\begin{align*}
   \MoveEqLeft \left|\int_0^T\Big\langle S[(1-\Delta_j)^{-1/2}\nabla_j]\cdot \left(\nabla_j\Big[\frac{1}{|x_j-x_k|} \cL_{I,j}u(\cdot,x)\Big]\right),v(t,x) \Big\rangle \;dt\right|\notag\\
   &\lesssim_{\alpha} \int_0^T\left\|\Big[\frac{1}{|x_j-x_k|^\alpha}\cL_{I,j} \nabla_j u(s,x)\Big]\right\|_{\cH}\left\|\frac{1}{|x_j-x_k|^{1-\alpha}} S^*\big((1-\Delta_j)^{-1/2}\nabla_jv(\cdot,x)\big)(s) \right\|_{\cH} \;ds.
\end{align*}
Proceeding as in \eqref{eq:SW-scalar}, we infer that under conditions 
\eqref{ass1}--\eqref{ass2} in Assumption~\ref{ass} and for $0<T<1$,
\begin{align}\label{case2}
    \left\|S\cL_I( |x_j-x_k|^{-1}u)\right\|_{X_{p,T}}\lesssim_{\alpha,p_,\widetilde{p}} T^{1/\theta'_{\widetilde{p}}-1/\theta_{p}} \|u\|_{X^1_{I,p,T}}
\end{align}

{\bf Case 3.} 
{\bf Case 3.} Finally, we consider the case $\{j,k\}\subset I$.  
In this case, we have
\begin{align*}
    \cL_{I}\frac{1}{|x_j-x_k|}u=(1-\Delta_j)^{1/2}(1-\Delta_k)^{1/2}\frac{1}{|x_j-x_k|}\cL_{I,j,k}u.
\end{align*}
According to \eqref{frac-laplace-dec}, we have
\begin{align}\label{eq:LI-W-dec}
  \MoveEqLeft  \cL_{I}\frac{1}{|x_j-x_k|}u=\frac{1}{(1-\Delta_j)^{1/2}}\frac{1}{(1-\Delta_k)^{1/2}}\frac{1}{|x_j-x_k|}\cL_{I,j,k}u\notag\\
    &\quad+ \frac{\nabla_j}{(1-\Delta_j)^{1/2}}\frac{1}{(1-\Delta_k)^{1/2}}\cdot \nabla_j\big[\frac{1}{|x_j-x_k|}\cL_{I,j,k}u\big]\notag\\
    &\quad+ \frac{1}{(1-\Delta_j)^{1/2}}\frac{\nabla_k}{(1-\Delta_k)^{1/2}}\cdot \nabla_k\big[\frac{1}{|x_j-x_k|}\cL_{I,j,k}u\big]\notag\\
    &\quad +\big[\frac{\nabla_j}{(1-\Delta_j)^{1/2}}\otimes \frac{\nabla_k}{(1-\Delta_k)^{1/2}}\big]\cdot \left(\nabla_j\otimes \nabla_k\big[\frac{1}{|x_j-x_k|}\cL_{I,j,k}u\big]\right).
\end{align}
Thus, the study of $S\cL_{I}(|x_j-x_k|^{-1}u)$ reduces to analyzing the four terms on the right-hand side of \eqref{eq:LI-W-dec}.

Concerning the first term, by \eqref{eq:SWu1} we obtain that, under conditions 
\eqref{ass1}--\eqref{ass2} in Assumption~\ref{ass}, for any $0<T<1$ and any 
$D\subset\{1,\cdots,N\}$ with $1\leq |D|\leq 2$,
\begin{align}\label{eq:4.25}
 \MoveEqLeft   \left\|S\frac{1}{(1-\Delta_j)^{1/2}}\frac{1}{(1-\Delta_k)^{1/2}}\Big(\frac{1}{|x_j-x_k|}\cL_{I,j,k}u\Big)\right\|_{X_{p,T}}\notag\lesssim_{p} \left\|S\Big(\frac{1}{|x_j-x_k|}u\Big)\right\|_{X_{p,T}}\\
 &\qquad\qquad\qquad\qquad\qquad\lesssim_{\alpha,p_,\widetilde{p}} T^{1/\theta'_{\widetilde{p}}-1/\theta_{p}}  \| \cL_{I,j,k}u\|_{X_{p,T}} \lesssim_{\alpha,p_,\widetilde{p}} T^{1/\theta'_{\widetilde{p}}-1/\theta_{p}}  \|\cL_{I}u\|_{X_{p,T}}.
\end{align}
Here, the first and last inequalities follow from \eqref{eq:2.6b}.

The second term on the right-hand side of \eqref{eq:LI-W-dec} can be studied as in 
\eqref{eq:4.20}, by replacing $a_{\mu}$ with $x_j$:
\begin{align*}
    \left\||x_j-x_k|^{1-\alpha}\nabla_j\Big[\frac{1}{|x_j-x_k|}\Big] 
    \cL_{I,j,k}u(\cdot,x)\right\|_{\cH}
    \;\lesssim\; \left\|\frac{1}{|x_j-x_k|^\alpha}\cL_{I,j,k}\nabla_j u\right\|_{\cH}.
\end{align*}
Proceeding as for \eqref{eq:SW-scalar} and using Theorem~\ref{th:2.7}, we obtain that, 
under conditions \eqref{ass1}--\eqref{ass2} in Assumption~\ref{ass}, for any $0<T<1$ 
and any $D\subset\{1,\cdots,N\}$ with $1\leq |D|\leq 2$, 
\begin{align*}
       \MoveEqLeft \left|\int_0^T\Big\langle S\frac{\nabla_j}{(1-\Delta_j)^{1/2}}\frac{1}{(1-\Delta_j)^{1/2}} \cdot \left(\nabla_j\Big[\frac{1}{|x_j-x_k|} \cL_{I,j,k}u(\cdot,x)\Big]\right),v(t,x) \Big\rangle \;dt\right|\notag\\
   &\lesssim_{\alpha} \int_0^T\left\|\Big[\frac{1}{|x_j-x_k|^\alpha}\cL_{I,j,k} \nabla_j u(s,x)\Big]\right\|_{\cH}\left\|\frac{1}{|x_j-x_k|^{1-\alpha}} S^*\left(\frac{\nabla_j}{(1-\Delta_j)^{1/2}}\frac{1}{(1-\Delta_j)^{1/2}} v(\cdot,x)\right)(s) \right\|_{\cH} \;ds\\
   &\lesssim_{\alpha,p,\widetilde{p}} T^{1/\theta'_{\widetilde{p}}-1/\theta_{p}}\Big(\|\cL_{I}u\|_{L_t^{\infty}(I_T,\cH)}+ \|\cL_{I}u\|_{L_t^{\theta_{p}}(I_T,L^{p,2}_{j,k})}\Big)\min\{\|v\|_{L_t^{\theta'_{p}}(I_T, L^{p',2}_D)},\|v\|_{L^1_t(I_T,\cH)}\}.
\end{align*}
Hence, under conditions \eqref{ass1}--\eqref{ass2} in Assumption~\ref{ass}, 
for any $0<T<1$ and any $D\subset\{1,\cdots,N\}$ with $1\leq |D|\leq 2$,
\begin{align}\label{eq:4.26}
 \MoveEqLeft   \left\|S\frac{\nabla_j}{(1-\Delta_j)^{1/2}}\frac{1}{(1-\Delta_k)^{1/2}}\Big(\frac{1}{|x_j-x_k|}\cL_{I,j,k}u\Big)\right\|_{X_{p,T}}\lesssim_{\alpha,p_,\widetilde{p}} T^{1/\theta'_{\widetilde{p}}-1/\theta_{p}}  \|\cL_{I}u\|_{X_{p,T}}.
\end{align}

Analogously, for the third term on the right-hand side of \eqref{eq:LI-W-dec}, we have that under condition \eqref{ass1}-\eqref{ass2} in Assumption \ref{ass} and for any $0<T<1$ and any $D\subset\{1,\cdots,N\}$ with $1\leq |D|\leq 2$,
\begin{align}\label{eq:4.27}
 \MoveEqLeft   \left\|S\frac{1}{(1-\Delta_j)^{1/2}}\frac{\nabla_k}{(1-\Delta_k)^{1/2}}\Big(\frac{1}{|x_j-x_k|}\cL_{I,j,k}u\Big)\right\|_{X_{p,T}}\lesssim_{\alpha,p_,\widetilde{p}} T^{1/\theta'_{\widetilde{p}}-1/\theta_{p}}  \|\cL_{I}u\|_{X_{p,T}}.
\end{align}

It remains to study the last term on the right-hand side of \eqref{eq:LI-W-dec}, 
where we will use Lemma~\ref{Hardy} together with \eqref{u-anti} 
(i.e., the antisymmetry of $u$ with respect to $\{j,k\}$). We have
\begin{align*}
    \nabla_j\otimes \nabla_k\Big[\frac{1}{|x_j-x_k|}\cL_{I,j,k}u\Big]&=\frac{1}{|x_j-x_k|} \nabla_j\otimes \nabla_k \cL_{I,j,k}u+\big[\nabla_j\frac{1}{|x_j-x_k|}\big]\otimes  \nabla_k \cL_{I,j,k}u\\
    &\quad+\big[\nabla_k\frac{1}{|x_j-x_k|}\big] \otimes \nabla_j \cL_{I,j,k}u+\big[\nabla_j\otimes \nabla_k \cL_{I,j,k}\frac{1}{|x_j-x_k|}\big] u.
\end{align*}
Thus, by Lemma \ref{Hardy1}, Lemma \ref{Hardy}, and the antisymmetry property \eqref{u-anti},
\begin{align*}
 \MoveEqLeft \left\||x_j-x_k|^{1-\alpha}\nabla_j\Big[\frac{1}{|x_j-x_k|}\Big] \cL_{I,j,k}u(\cdot,x)\right\|_{\cH}\lesssim\left\|\frac{1}{|x_j-x_k|^{\alpha}}\cL_{I,j,k}\nabla_j\otimes \nabla_k u\right\|_{\cH}+\left\|\frac{1}{|x_j-x_k|^{1+\alpha}}\cL_{I,j,k}\nabla_j u\right\|_{\cH}\\
     &\qquad\qquad\qquad\qquad+ \left\|\frac{1}{|x_j-x_k|^{1+\alpha}}\cL_{I,j,k}\nabla_k u\right\|_{\cH}+\left\|\frac{1}{|x_j-x_k|^{2+\alpha}}\cL_{I,j,k} u\right\|_{\cH}\lesssim \left\|\frac{1}{|x_j-x_k|^{\alpha}}\cL_{I,j,k}\nabla_j\otimes \nabla_k u\right\|_{\cH}.
\end{align*}
Proceeding as in \eqref{eq:SW-scalar} and using Theorem \ref{th:2.7}, we infer that, 
under conditions \eqref{ass1}--\eqref{ass2} in Assumption \ref{ass}, 
for any $0<T<1$ and any $D\subset\{1,\cdots,N\}$ with $1\leq |D|\leq 2$,
\begin{align*}
\MoveEqLeft \left|\int_0^T\Big\langle S\Big[\frac{\nabla_j}{(1-\Delta_j)^{1/2}}\otimes\frac{\nabla_k}{(1-\Delta_k)^{1/2}}\Big] \cdot \left(\nabla_j\otimes\nabla_k\Big[\frac{1}{|x_j-x_k|} \cL_{I,j,k}u(\cdot,x)\Big]\right),v(t,x) \Big\rangle \;dt\right|\notag\\
&\lesssim_{\alpha} \int_0^T\left\|\frac{1}{|x_j-x_k|^\alpha}\cL_{I,j,k} \nabla_j\otimes \nabla_k u(s,x)\right\|_{\cH} \left\|\frac{1}{|x_j-x_k|^{1-\alpha}} S^*\left(\frac{\nabla_j}{(1-\Delta_j)^{1/2}}\otimes \frac{\nabla_k}{(1-\Delta_k)^{1/2}} v(\cdot,x)\right)(s) \right\|_{\cH} \;ds\\
   &\lesssim_{\alpha,p,\widetilde{p}} T^{1/\theta'_{\widetilde{p}}-1/\theta_{p}}\Big(\|\cL_{I}u\|_{L_t^{\infty}(I_T,\cH)}+ \|\cL_{I}u\|_{L_t^{\theta_{p}}(I_T,L^{p,2}_{j,k})}\Big)\min\{\|v\|_{L_t^{\theta'_{p}}(I_T, L^{p',2}_D)},\|v\|_{L^1_t(I_T,\cH)}\}.
\end{align*}
As a result, under condition \eqref{ass1}-\eqref{ass2} in Assumption \ref{ass} and for any $0<T<1$ and any $D\subset\{1,\cdots,N\}$ with $1\leq |D|\leq 2$,
\begin{align}\label{eq:4.28}
 \MoveEqLeft   \left\|S\Big[\frac{\nabla_j}{(1-\Delta_j)^{1/2}}\otimes\frac{\nabla_k}{(1-\Delta_k)^{1/2}}\Big] \cdot \left(\nabla_j\otimes\nabla_k\Big[\frac{1}{|x_j-x_k|} \cL_{I,j,k}u(\cdot,x)\Big]\right)\right\|_{X_{p,T}}\lesssim_{\alpha,p_,\widetilde{p}} T^{1/\theta'_{\widetilde{p}}-1/\theta_{p}}  \|\cL_{I}u\|_{X_{p,T}}.
\end{align}
Now we conclude from \eqref{eq:4.25}-\eqref{eq:4.28} that under condition \eqref{ass1}-\eqref{ass2} in Assumption \ref{ass} and for any $0<T<1$,
\begin{align}\label{case3}
   \left\|S\Big(\frac{1}{|x_j-x_k|} u(\cdot,x)\Big)\right\|_{X^1_{I,p,T}}\lesssim_{\alpha,p_,\widetilde{p}} T^{1/\theta'_{\widetilde{p}}-1/\theta_{p}}  \|u\|_{X^1_{I,p,T}}.
\end{align}

{\bf Conclusion for all cases.} Finally we can conclude from \eqref{case1}, \eqref{case2} and \eqref{case3} that  under condition \eqref{ass1}-\eqref{ass2} in Assumption \ref{ass} and for any $0<T<1$,
\begin{align}\label{eq:LI-SW-tatal}
   \left\|S\Big(\frac{1}{|x_j-x_k|} u(\cdot,x)\Big)\right\|_{X^1_{I,p,T}}\lesssim_{\alpha,p_,\widetilde{p}} ZT^{1/\theta'_{\widetilde{p}}-1/\theta_{p}}  \|u\|_{X^1_{I,p,T}}.
\end{align}

{\bf Step 3. Conclusion.} From \eqref{eq:LI-SV-total} and \eqref{eq:LI-SW-tatal}, we infer that for any $0<T<1$, $p$ satisfying \eqref{eq:p1} and $\widetilde{p}$ satisfying \eqref{eq:tildep1}, there exists a constant $C_{T,2}:=C_{T,2}(\alpha,p,\widetilde{p})\geq 1$ such that
\begin{align}\label{eq:CT2}
    \left\|Qu\right\|_{X^1_{I,p,T}}\leq C_{T,2} (Z+N)NT^{1/\theta'_{\widetilde{p}}-1/\theta_{p}} \|u\|_{X^1_{I,p,T}}.
\end{align}
Now let $C_{T,2}(Z+N)NT^{1/\theta'_{\widetilde{p}}-1/\theta_{p}}\leq \frac{1}{2}$, we get $ \|Qu\|_{X^1_{I,p,T}}\leq \frac{1}{2}\|u\|_{X^1_{I,p,T}}.$ Thus under Assumption \ref{ass}, we have that $1+iQ$ is invertible on $X^1_{I,p,T}$. As a result,
\begin{align*}
    \|u\|_{X^1_{I,p,T}}=\|(1+iQ)^{-1}(U_0(\cdot)u_0)\|_{X^1_{I,p,T}}\leq 2\|U_0(\cdot)u_0\|_{X^1_{I,p,T}}\lesssim_{q} \|u_0\|_{H_{I,\rm mix}^1}.
\end{align*}
This gives \eqref{eq:u-mix} and shows the uniqueness of the solution $u$ in $X_{I,p,T}^1$. Hence the theorem.
\end{proof}

\section{Justification of the sparse grid approximation \texorpdfstring{\eqref{hyper-approx}}{}}\label{sec:5} 
Now we are going to prove Theorem \ref{th:justif}, in particular \eqref{error-bound}. Before going further, we need the following result which can be regarded as an evolution version of \eqref{eq:approx1}.
\begin{lemma}\label{lem:5.3}
Let $u_0\in H_{I_1,\rm mix}^1\bigcap H_{I_2,\rm mix}^1$, and $\mathcal{P}_R$ be a projector satisfying Assumption~\ref{ass:1.8}. Under Assumption \ref{ass} we have
    \begin{align}\label{eq:u-Pu}
        \|(1-\mathcal{P}_{R})u\|_{X_{p,T}}\lesssim_p \frac{1}{R}\sum_{\ell=1,2}\|u_0\|_{H_{I_\ell,\rm mix}^1}.
    \end{align}
\end{lemma}
\begin{proof}
By Duhamel formula\eqref{duhamel}, we have that 
\[
(1-\mathcal{P}_{R})u(t)=(1-\mathcal{P}_{R})U_0(t)u_0-i(1-\mathcal{P}_{R})Qu(t).
\]
Thus, by \eqref{eq:2.2a},
\begin{align*}
    \|(1-\mathcal{P}_{R})u\|_{X_{p,T}}&\leq \|U_0(t) (1-\mathcal{P}_{R})u_0\|_{X_{p,T}}+\|(1-\mathcal{P}_{R})Qu(t)\|_{X_{p,T}}\lesssim_{p} \|(1-\mathcal{P}_{R})u_0\|_{\cH}+\|(1-\mathcal{P}_{R})Qu(t)\|_{X_{p,T}}.
\end{align*}
According to \eqref{eq:3.5}, we have
\begin{align}\label{eq:ope-R}
    \|(1-\mathcal{P}_{R})u_0\|_{\cH}\leq \frac{1}{R}\big\|\sum_{\ell=1,2} \cL_{I_\ell}u_0\big\|_{\cH}\leq \frac{1}{R}\sum_{\ell=1,2}\|u\|_{H_{I_\ell,\rm mix}^1}.
\end{align}
Thus, using \eqref{eq:2.2a},
\begin{align}\label{eq:5.6}
    \|(1-\mathcal{P}_{R})u\|_{X_{p,T}}\lesssim_{p} R^{-1}\sum_{\ell=1,2}\|u_0\|_{H_{I_\ell,\rm mix}^1}+\|(1-\mathcal{P}_{R})Qu(t)\|_{X_{p,T}}.
\end{align}
It remains to show that $\|(1-\mathcal{P}_{R})Qu(t)\|_{X_{p,T}}\lesssim_p \frac{1}{R}\sum_{\ell=1,2}\|u_0\|_{H_{I_\ell,\rm mix}^1}$. To do so, here we are going to prove
\begin{align}\label{eq:5.7}
    \|(1-\mathcal{P}_{R})Qu(t)\|_{X_{p,T}}\lesssim_p \frac{1}{R}\sum_{\ell=1,2}\|u(t)\|_{X^1_{I_\ell,p,T}}.
\end{align}
Then by \eqref{eq:u-mix}, under Assumption \ref{ass} we have
\begin{align*}
    \|(1-\mathcal{P}_{R})Qu(t)\|_{X_{p,T}}\lesssim_p \frac{1}{R}\sum_{\ell=1,2}\|u_0\|_{H_{I_\ell,\rm mix}^1}.
\end{align*}
This and \eqref{eq:5.6} give \eqref{eq:u-Pu}. Hence this lemma.

To end the proof, we prove \eqref{eq:5.7} by using Corollary \ref{cor:2.5}. 
As in the proof of mixed regularity, we consider this problem in the scalar product: 
for any $D\subset\{1,\cdots,N\}$ with $1\leq |D|\leq 2$, and for any 
$u(t,x)\in X^1_{I,p,T}$ and $v(t,x)\in L_t^{\theta'_{p}}(I_T,L^{p',2}_{D})$ 
or $v(t,x)\in L_t^{\infty}(I_T,\cH)$, we aim to show 
\begin{align*}
   \left|\int_{0}^T\langle (1-\mathcal{P}_{R})Qu(t),v(t) \rangle \;dt\right|\lesssim T^\theta \sum_{\ell=1,2}\|u\|_{X^1_{I_\ell,p,T}}\min\{\|v\|_{L_t^{\theta'_{p}}(I_T,L^{p',2}_{D})},\|v\|_{L^\infty_t(I_T,\cH)}\}
\end{align*}
for some $\theta>0$ and $p>2$. Indeed, we have
\begin{align}\label{eq:5.8}
 \MoveEqLeft   \left|\int_{0}^T\langle (1-\mathcal{P}_{R})Qu(t),v(t) \rangle \;dt\right|\leq \sum_{\ell=1,2}\left|\int_{0}^T\left\langle \cL_{I_\ell}V(s,x_j) u(s),S^* \Big[\Big(\sum_{\ell=1,2}\cL_{I_\ell}\Big)^{-1}(1-\mathcal{P}_{R})v\Big](s) \right\rangle \;ds\right|\notag\\
    &\qquad\qquad\qquad\qquad\qquad\qquad+\sum_{\substack{\ell=1,2\\ 1\leq j<k\leq N}}\left|\int_{0}^T\left\langle \cL_{I_\ell}W(x_j,x_k) u(s),S^* \Big[\Big(\sum_{\ell=1,2}\cL_{I_\ell}\Big)^{-1}(1-\mathcal{P}_{R})v\Big](s) \right\rangle \;ds\right|.
\end{align}
The proof of the terms on the right-hand side of \eqref{eq:5.8} 
is essentially the same as in Step~1 and Step~2 of the proof of Theorem~\ref{th:mixregularity}: 
we only need to replace the Strichartz estimate \eqref{eq:2.2c} used there 
by the Strichartz estimate \eqref{eq:2.3b}. 
Then, under Assumption~\ref{ass}, we obtain \eqref{eq:5.7}.
\end{proof}

\begin{proof}[Proof of Theorem \ref{th:justif}] 
We first consider the existence of solutions to \eqref{hyper-approx}. 
The proof is essentially the same as for Theorem~\ref{th:existence}, 
with one modification: the Strichartz estimate \eqref{eq:2.2c} used there 
should be replaced by \eqref{eq:2.3a}, as in Lemma~\ref{lem:5.3}. 
It then follows that for every $u_0\in \cH$, the problem \eqref{hyper-approx} has 
a unique global-in-time solution $u_R\in X_{p,\infty}$. 
Moreover, by \eqref{duhamel} and the identity $\mathcal{P}_R^2=\mathcal{P}_R$, 
we have $\mathcal{P}_R u_R=u_R$.  

We now turn to the error estimate \eqref{error-bound}. 
Observe that
\begin{align}\label{eq:5.9}
    \|u-u_R\|_{X_{p,T}}
    \leq \|u-\mathcal{P}_{R}u\|_{X_{p,T}}
         +\|\mathcal{P}_{R}u-u_R\|_{X_{p,T}}.
\end{align}
It remains to control the difference $\mathcal{P}_{R}u-u_R$. 
By the Duhamel formula \eqref{duhamel}, we have
\[
    \mathcal{P}_{R}u-u_R = -\,i\,\mathcal{P}_{R}Q(u-u_R).
\]
Replacing the Strichartz estimate \eqref{eq:2.2c} by \eqref{eq:2.3a} 
and proceeding as for Step~1 and Step~2 in the proof of Theorem~\ref{th:existence}, 
we deduce that, under conditions \eqref{ass1}–\eqref{ass2} in Assumption~\ref{ass} 
and for $0<T<1$, there exists a constant 
$C_{T,3}:=C_{T,3}(\alpha,p,\widetilde{p})\geq 1$ such that
\begin{align}\label{eq:CT3}
    \|\mathcal{P}_{R}Q(u-u_R)\|_{X_{p,T}}
    \leq C_{T,3}(Z+N)N\,T^{1/\theta'_{\widetilde{p}}-1/\theta_{p}}
        \,\|u-u_R\|_{X_{p,T}}.
\end{align}
Hence, under Assumption~\ref{ass},
\[
    \|\mathcal{P}_{R}u-u_R\|_{X_{p,T}}
    = \|\mathcal{P}_{R}Q(u-u_R)\|_{X_{p,T}}
    \leq \tfrac{1}{2}\|u-u_R\|_{X_{p,T}}.
\]
Substituting this bound into \eqref{eq:5.9}, we obtain
\[
    \|u-u_R\|_{X_{p,T}}
    \leq \|u-\mathcal{P}_{R}u\|_{X_{p,T}}
         + \tfrac{1}{2}\|u-u_R\|_{X_{p,T}}.
\]
Therefore, by Lemma~\ref{lem:5.3}, under Assumption \ref{ass}, 
\[
     \|u-u_R\|_{X_{p,T}}
     \leq 2\|u-\mathcal{P}_{R}u\|_{X_{p,T}}
     \;\lesssim_p\; \frac{1}{R}\sum_{\ell=1,2}\|u_0\|_{H^1_{I_\ell,\mathrm{mix}}}.
\]
This proves \eqref{error-bound} and completes the proof. 
\end{proof}

\noindent\textbf{Acknowledgments.} 
We thank Huajie Chen and Harry Yserentant for valuable discussions and insightful feedback. DZ's work was partially supported by NSFC 124B2020.

\medskip
\bibliographystyle{plain}
\bibliography{reference}

\begin{thebibliography}{10}

\bibitem{baiardi2021electron}
A.~Baiardi.
\newblock Electron dynamics with the time-dependent density matrix renormalization group.
\newblock {\em J. Chem. Theory Comput.}, 17(6):3320--3334, 2021.

\bibitem{Catto}
C.~Bardos, I.~Catto, N.~Mauser, and S.~Trabelsi.
\newblock Setting and analysis of the multi-configuration time-dependent {H}artree-{F}ock equations.
\newblock {\em Arch. Ration. Mech. Anal.}, 198(1):273--330, 2010.

\bibitem{Griebel2}
H.-J. Bungartz and M.~Griebel.
\newblock Sparse grids.
\newblock {\em Acta Numer.}, 13:147--269, 2004.

\bibitem{Cances}
E.~Canc\`es and C.~Le~Bris.
\newblock On the time-dependent {H}artree-{F}ock equations coupled with a classical nuclear dynamics.
\newblock {\em Math. Models Methods Appl. Sci.}, 9(7):963--990, 1999.

\bibitem{Chadam}
J.~M. Chadam.
\newblock The time-dependent {H}artree-{F}ock equations with {C}oulomb two-body interaction.
\newblock {\em Comm. Math. Phys.}, 46(2):99--104, 1976.

\bibitem{ChristKiselev}
M.~Christ and A.~Kiselev.
\newblock Maximal functions associated to filtrations.
\newblock {\em J. Funct. Anal.}, 179(2):409--425, 2001.

\bibitem{moleculardynamics-Hans}
H.~Feldmeier and J.~Schnack.
\newblock Molecular dynamics for fermions.
\newblock {\em Rev. Mod. Phys.}, 72(3):655--688, 2000.

\bibitem{Griebel}
M.~Griebel and J.~Hamaekers.
\newblock A wavelet based sparse grid method for the electronic {S}chr\"{o}dinger equation.
\newblock In {\em International {C}ongress of {M}athematicians. {V}ol. {III}}, pages 1473--1506. Eur. Math. Soc., Z\"{u}rich, 2006.

\bibitem{Keel}
M.~Keel and T.~Tao.
\newblock Endpoint {S}trichartz estimates.
\newblock {\em Amer. J. Math.}, 120(5):955--980, 1998.

\bibitem{Yserentant3}
H.-C. Kreusler and H.~Yserentant.
\newblock The mixed regularity of electronic wave functions in fractional order and weighted {S}obolev spaces.
\newblock {\em Numer. Math.}, 121(4):781--802, 2012.

\bibitem{quantumdynamics-LasserLubich}
Caroline Lasser and Christian Lubich.
\newblock Computing quantum dynamics in the semiclassical regime.
\newblock {\em Acta Numer.}, 29:229--401, 2020.

\bibitem{li2020real}
X.~Li, N.~Govind, C.~Isborn, A.~E.~III DePrince, and K.~Lopata.
\newblock Real-time time-dependent electronic structure theory.
\newblock {\em Chem. Rev.}, 120(18):9951--9993, 2020.

\bibitem{singular}
S.~Lu, Y.~Ding, and D.~Yan.
\newblock {\em Singular integrals and related topics}.
\newblock World Scientific Publishing Co. Pte. Ltd., Hackensack, NJ, 2007.

\bibitem{quantumtoclassical-lubich}
Christian Lubich.
\newblock {\em From quantum to classical molecular dynamics: reduced models and numerical analysis}.
\newblock Zurich Lectures in Advanced Mathematics. European Mathematical Society (EMS), Z\"urich, 2008.

\bibitem{Abinitio}
Dominik Marx and J\"urg Hutter.
\newblock {\em Ab Initio Molecular Dynamics Basic Theory and Advanced Methods}.
\newblock Cambridge University Press, 2009.

\bibitem{meng}
L.~Meng.
\newblock On the mixed regularity of {$N$}-body {C}oulombic wavefunctions.
\newblock {\em ESAIM Math. Model. Numer. Anal.}, 57(4):2257--2282, 2023.

\bibitem{schlag}
C.~Muscalu and W.~Schlag.
\newblock {\em Classical and multilinear harmonic analysis. {V}ol. {I}}, volume 137 of {\em Cambridge Studies in Advanced Mathematics}.
\newblock Cambridge University Press, Cambridge, 2013.

\bibitem{nys2024ab}
J.~Nys, G.~Pescia, A.~Sinibaldi, and G.~Carleo.
\newblock Ab-initio variational wave functions for the time-dependent many-electron {S}chr\"odinger equation.
\newblock {\em Nat. Commun.}, 15(1):9404, 2024.

\bibitem{tao}
T.~Tao.
\newblock {\em Nonlinear dispersive equations: Local and global analysis}, volume 106.
\newblock American Mathematical Society, 2006.

\bibitem{tuckerman_2002}
M.~E. Tuckerman.
\newblock Ab initio molecular dynamics: basic concepts, current trends and novel applications.
\newblock {\em J. Phys. Condens. Matter}, 14(50):R1297--R1355, 2002.

\bibitem{wozniak2023exploring}
A.~P. Wo{\'z}niak, M.~Lewenstein, and R.~Moszy{\'n}ski.
\newblock Exploring the attosecond laser-driven electron dynamics in the hydrogen molecule with different real-time time-dependent configuration interaction approaches.
\newblock In M.~Musia{\l} and I.~Grabowski, editors, {\em Polish Quantum Chemistry from Ko{\l}os to Now}, volume~87 of {\em Adv. Quantum Chem.}, pages 167--190. 2023.

\bibitem{Yajima1}
K.~Yajima.
\newblock Existence of solutions for {S}chr\"{o}dinger evolution equations.
\newblock {\em Comm. Math. Phys.}, 110(3):415--426, 1987.

\bibitem{Yajima2}
K.~Yajima.
\newblock Existence and regularity of propagators for multi-particle {S}chr\"{o}dinger equations in external fields.
\newblock {\em Comm. Math. Phys.}, 347(1):103--126, 2016.

\bibitem{Yserentant1}
H.~Yserentant.
\newblock On the regularity of the electronic {S}chr\"{o}dinger equation in {H}ilbert spaces of mixed derivatives.
\newblock {\em Numer. Math.}, 98(4):731--759, 2004.

\bibitem{Yserentant-sparegrid}
H.~Yserentant.
\newblock Sparse grid spaces for the numerical solution of the electronic {S}chr\"odinger equation.
\newblock {\em Numer. Math.}, 101(2):381--389, 2005.

\bibitem{Yserentant2}
H.~Yserentant.
\newblock The hyperbolic cross space approximation of electronic wavefunctions.
\newblock {\em Numer. Math.}, 105(4):659--690, 2007.

\bibitem{Yserentant4}
H.~Yserentant.
\newblock {\em Regularity and approximability of electronic wave functions}, volume 2000 of {\em Lecture Notes in Mathematics}.
\newblock Springer-Verlag, Berlin, 2010.

\end{thebibliography}
\end{document}